\documentclass{amsart}
\pdfoutput=1

\usepackage{amsmath, amsthm}
\usepackage{amssymb, latexsym}
\usepackage{amsfonts}
\usepackage{graphicx}
\usepackage{hyperref}

\usepackage[all]{xy}

\usepackage{subfigure}

\renewcommand{\labelenumi}{(\arabic{enumi})}

\newcommand{\enumia}{\renewcommand{\labelenumi}{(\arabic{enumi})}}
\newcommand{\enumir}{\renewcommand{\labelenumi}{(\roman{enumi})}}

\theoremstyle{plain}
\newtheorem{thm}{Theorem}[section]

\newtheorem{conj}[thm]{Conjecture}
\newtheorem{cor}[thm]{Corollary}
\newtheorem{lem}[thm]{Lemma}

\newtheorem{prop}[thm]{Proposition}
\newtheorem*{thmpolygonal}{Theorem \ref{thm:polygonal}}
\newtheorem*{corpolygonal}{Corollary \ref{cor:polygonal}}

\newtheorem*{conjtiling}{Conjecture \ref{conj:tiling}}

\newtheorem*{nonprop}{Proposition \ref{prop:exotic surface}}

\theoremstyle{definition}

\newtheorem{defn}[thm]{Definition}
\newtheorem{que}[thm]{Question}
\newtheorem{rem}[thm]{Remark}
\newtheorem{exmp}[thm]{Example}

\def\co{\colon\thinspace}
\def\cobar{|\thinspace}

\newcommand{\cay}{\ensuremath{\mathrm{Cay}}}
\newcommand{\sign}{\ensuremath{\mathrm{sign}}}
\renewcommand{\mod}{\ensuremath{\ (\mathrm{mod}\ }}

\newcommand{\bbb}[1]{\ensuremath{\mathbb{#1}}}

\newcommand{\Q}{\bbb{Q}}
\newcommand{\R}{\bbb{R}}

\newcommand{\Z}{\bbb{Z}}

\newcommand{\script}[1]{\ensuremath{\mathcal{#1}}}
\def\AA{\script{A}}
\def\BB{\script{B}}
\def\GG{\script{G}}

\def\KK{\script{K}}
\def\UU{{\script U}}
\def\TT{{\script T}}

\newcommand{\smallcaps}[1]{\textrm{\textsc{#1}}}

\newcommand{\sym}{\smallcaps{Sym}}
\renewcommand{\int}{\mathrm{int}}

\newcommand{\ba}{\begin{array}}
\newcommand{\ea}{\end{array}}
\newcommand{\be}{\begin{enumerate}}
\newcommand{\ee}{\end{enumerate}}
\newcommand{\bd}{\begin{defn}}
\newcommand{\ed}{\end{defn}}
\newcommand{\bi}{\begin{itemize}}
\newcommand{\ei}{\end{itemize}}

\begin{document}
\title{Polygonal Words in Free Groups}
\author{Sang-hyun Kim}
\address{Department of Mathematical Sciences, KAIST, 335 Gwahangno, Yuseong-gu, Daejeon 305-701, Republic of Korea}
\email{shkim@kaist.edu}
\thanks{Kim supported by Basic Science Research Program (2010-0023515) and Mid-career Researcher Program (2010-0027001) through the National Research Foundation of Korea (NRF) funded by the MEST}

\author{Henry Wilton}
\address{Mathematics 253-37, Caltech, Pasadena, CA 91125 \\ USA}
\email{wilton@caltech.edu}
\thanks{Wilton partially supported by NSF grant DMS-0906276}

\maketitle
\date{\today}

\subjclass[2000]{ Primary 20F36, 20F65}
\keywords{hyperbolic group, surface group}

\begin{abstract}    
A longstanding question of Gromov asks whether every one-ended word-hyperbolic group contains a subgroup isomorphic to the fundamental group of a closed hyperbolic surface.  An infinite family of word-hyperbolic groups can be obtained by taking doubles of free groups amalgamated along words that are not proper powers.  We define the set of polygonal words in a free group of finite rank, and prove that polygonality of the amalgamating word guarantees that the associated square complex virtually contains a $\pi_1$-injective closed surface.  We provide many concrete examples of classes of polygonal words.  For instance, in the case when the rank is two, we establish polygonality of words without an isolated generator, and also of almost all simple height-one words, including the Baumslag--Solitar relator $a^p (a^q)^b$ for $pq\ne0$.
\end{abstract}  
        
\section{Introduction}\label{sec:introduction}

A group $G$ is called a \textit{surface group} if $G$ is isomorphic to the fundamental group of a closed surface $S$ with non-positive Euler characteristic; if the Euler characteristic of $S$ is negative, we say that $G$ is a \textit{hyperbolic surface group}.  Motivated by the famous Surface Subgroup Conjecture in 3-manifold topology, Gromov asked the following question.\footnote{Kahn and Markovic have recently announced a positive resolution of the Surface Subgroup Conjecture.}

\begin{que}[Gromov]\label{que:Gromov}
	Does every one-ended word-hyperbolic group contain a hyperbolic surface group?
\end{que}

In this paper, we introduce a new criterion to examine this question in the context of doubles of free groups.  We denote by $F$ the free group generated by $\{a_1,\ldots,a_n\}$ for some $n>1$. When the rank needs to be explicitly shown, we also denote $F$ as $F_n$. We conventionally write the generators of $F_2$ as $a$ and $b$. Let $w=v_1 \cdots v_l$ be a word in $F$ where $v_i\in\{a_1^{\pm1},\ldots,a_n^{\pm1}\}$. Each $v_i$ is called a \textit{letter} of $w$. We say that $w$ is \textit{cyclically reduced} if $v_i\ne v_{i+1}^{-1}$ for each $i$, where the indices are taken modulo $l$. If $w$ is cyclically reduced, we call $l$ the \textit{length} of $w$ and write $|w|=l$.   The word $w\in F$ is called \textit{diskbusting} if $\{w\}$ does not belong to a proper free factor of $F$.

An infinite family of one-ended word-hyperbolic groups can be obtained by taking the \textit{double} of $F$ along $w$, namely $D(w)=F\ast_{\langle w\rangle} F$ where $w\in F$. The group $D(w)$ is one-ended if and only if $w$ is diskbusting.   Also, $D(w)$ is word-hyperbolic if and only if $w$ is not a proper power~\cite{Bestvina:1992p456}.  The main result of~\cite{Calegari:2008p1810} implies that if $D(w)$ is word-hyperbolic and $H_2(G;\Q)\ne0$, then $D(w)$ contains a hyperbolic surface group. Gordon and the second author proved a converse: the existence of a hyperbolic surface subgroup in $D(w)$ implies that $H_2(G;\Q)\ne0$ for some finite-index subgroup $G$ of $D(w)$~\cite{Gordon:2009p360}. Using a certain Eilenberg--Mac~Lane space $X(w)$ of $D(w)$ and looking at certain finite covers of $X(w)$, they also gave several sufficient conditions for $D(w)$ to have a finite-index subgroup $G$ with $H_2(G;\Q)\ne0$; see Section~\ref{sec:graph of graphs} for a precise definition of $X(w)$.

We will consider a combinatorial property, called \textit{polygonality} of a word in $F$.  A cyclically reduced word $w \in F$ is \textit{polygonal} if 
there exists a surface $S$ obtained by a side-pairing on polygonal disks $P_1,\ldots,P_m$ where each edge in $S^{(1)}$ is oriented and labeled by $\{a_1,\ldots,a_n\}$ in such a way that the image of each loop $\partial P_i$ in $S$ reads a non-trivial power of $w$ and moreover, there do not exist two incoming edges or two outgoing edges of the same label at each vertex in $S^{(0)}$; see Definition~\ref{defn:polygonal}.  We say that a CW--complex $X$ \textit{virtually} contains a closed hyperbolic surface, if there is a homeomorphic embedding of a closed hyperbolic surface into a finite cover of $X$.

\begin{thmpolygonal}
Let $w\in F$ be a cyclically reduced word.
\be
\item
If $w\in F$ is polygonal, then $X(w)$ virtually contains 
a closed hyperbolic surface.
\item
If a closed hyperbolic surface
embeds into a covering space of $X(w)$ with finitely generated fundamental group, then $w$ is polygonal.
\ee
\end{thmpolygonal}

Two words in $F$ are \textit{equivalent} if there exists an automorphism of $F$ mapping one to the other.
Theorem~\ref{thm:polygonal} will provide a way to detect surface subgroups in $D(w)$.

\begin{corpolygonal}
	If $w\in F$ is equivalent to a polygonal word, then $D(w)$ contains a surface group. 
\end{corpolygonal}

There are many examples of polygonal words. The following theorem summarizes the words in $F_2=\langle a,b\rangle$ that we prove are polygonal. We use the notation $u^v=v^{-1}uv$ for words $u,v$. A word in $F_2$ of the form $w = \prod_{i=1}^l a^{p_i}(a^{q_i})^b$ is called a \textit{simple height-one word} if all the integers $p_i$ are of the same sign, all the integers $q_i$ are of the same sign, and $l>0$.

\begin{thm}\label{thm:summary}
The following words in $F_2$ are polygonal.
\be
\item 	$w=\prod_i a^{p_i}b^{q_i}$ where $|p_i|>1$ and $|q_i|>1$ for each $i$ (Theorem~\ref{thm:no isolated 2}).
\item 	$w=\prod_i a^{p_i}b^{q_i}$ where $|p_i|>1$ and $|q_i|=1$ for each $i$, and  $\sum_i (\sign(p_i q_i)+\sign(p_{i+1}q_i))=0$ (Theorem~\ref{thm:no isolated 3}).
\item 	$w=\prod_i a^{p_{i}} (a^{q_i})^b$ where $|p_i|>1$ and $|q_i|>1$ for each $i$  (Corollary~\ref{cor:no isolated 3}).
\item 	A simple height-one word $w=\prod_i a^{p_i} (a^{q_i})^b$ where $pp'\le q^2$ and  $qq'\le p^2$; here, we denote $p=\sum_i p_i, q=\sum_i q_i, p' = \sum_{|p_i|=1}1$ and  $q'= \sum_{|q_i|=1}1$. (Theorem~\ref{thm:simple height 1}).
\ee
\end{thm}

Theorem~\ref{thm:summary} (4) implies that \textit{almost all} simple height-one words are polygonal, in a suitable sense (Remark~\ref{rem:simple height 1}).

A word $w\in F$ is \textit{minimal} if no equivalent word to $w$ is shorter.  If one establishes polygonality (up to minimizing the length) for all diskbusting  words then Question~\ref{que:Gromov} would be answered affirmatively for doubles of free groups.

\begin{conjtiling}[Tiling Conjecture]
A minimal diskbusting word is polygonal.
\end{conjtiling}

The hypothesis of minimality is necessary.  There are examples of non-minimal diskbusting words that are not polygonal.  In one sense, this is because surface subgroups need not manifest themselves as topological surfaces.

\begin{nonprop}
There exists a compact, non-positively curved square complex $Y$ such that $\pi_1(Y)$ is a hyperbolic surface group but such that no subspace of $Y$ is homeomorphic to a closed surface.
\end{nonprop}

\vspace{.1in}
\textbf{Note.}
As this work was being written up, it came to our attention that
N.~Brady, M.~Forester and E.~Martinez-Pedroza have also, independently, developed
a method of finding surface subgroups in certain graphs of free groups with cyclic edge groups.  They
use a theorem essentially the same as Theorem~\ref{thm:no isolated 2}.

\textbf{Acknowledgement.} Both authors are deeply grateful for the numerous inspirations from Professor Cameron Gordon. Also, they are thankful for the guidance of Professor Alan Reid.

\section{Surfaces in a Graph of Graphs with Cylindrical Edge Spaces}\label{sec:graph of graphs}
A \textit{graph} means a $1$--dimensional CW--complex. For a graph $\Gamma$, $V(\Gamma)$ and $E(\Gamma)$ will denote the vertex set and the edge set of $\Gamma$, respectively. If $\Gamma_1$ and $\Gamma_2$ are graphs, a continuous map $\phi\co \Gamma_1\to\Gamma_2$ sending each vertex to a vertex and each edge to to an edge is called a \textit{graph map}. An \textit{immersion} is a graph map which  is locally injective. $\cay(F)/F$ denotes the bouquet of $n$ circles $\gamma_1,\ldots,\gamma_n$ so that $\cay(F)/F$ is an Eilenberg--Mac~Lane space for $F$. We always assume that each circle $\gamma_i\co S^1\to\cay(F)/F$ is equipped with an orientation and a label by $a_i$, and we identify $F=\pi_1(\cay(F)/F)$ by letting $[\gamma_i]=a_i$. 
We say that a loop $\gamma\co S^1\to\cay(F)/F$ \textit{reads} $w=a_{i_1}^{m_1}a_{i_2}^{m_2}\cdots a_{i_l}^{m_l}\in F$ where $m_i\in\Z$, 
if $\gamma$ is homotopic to the concatenation $\gamma_{i_1}^{m_1}\gamma_{i_2}^{m_2}\cdots \gamma_{i_l}^{m_1}$
as a based loop. 
For any cyclically reduced word $w\in F$, there exists an immersion $\gamma_w\co S^{1}\to \cay(F)/F$ reading $w$. Any immersion $\phi$ from a graph $\Gamma$ to $\cay(F)/F$ induces a labeling and an orientation of each edge in $\Gamma$, so that $\phi$ preserves the labeling and the orientation. An edge in $\Gamma$ labeled by a generator $a_i$ is called an \textit{$a_i$-edge}. For $F'\le F$, $\cay(F)/F' $ will denote the coset Cayley graph of $F/F'$; in other words, $\cay(F)/F' $ is the cover of $\cay(F, F)$ corresponding to $F'$.

We also consider an \textit{abstract graph} in Scott--Wall's sense \cite{SW1979}; that is, an abstract graph $\Gamma$ consists of the vertex set $V(\Gamma)$, the edge set $E(\Gamma)$, the origin map $\iota\co E(\Gamma)\to V(\Gamma)$, and a fixed-point-free involution on $E(\Gamma)$ sending $e$ to $\bar e$. Note that $|E(\Gamma)|$ is twice the number of the edges in the geometric realization of $\Gamma$ as a 1--dimensional CW--complex.
   
\begin{defn}
\be
\item
Let $\Gamma$ be an abstract graph, $\{G_v\cobar v\in V(\Gamma)\}$ be a collection of finitely generated free groups and $\{w_e\cobar e\in E(\Gamma)\}$ be a collection of non-trivial words such that $w_e\in G_{\iota(e)}$. We call the tuple $\GG = (\Gamma, \{G_v\cobar v\in V(\Gamma)\}, \{w_e\cobar e\in E(\Gamma)\})$ a \textit{graph of free groups with cyclic edge groups}.
\item
Let $\GG = (\Gamma, \{G_v\cobar v\in V(\Gamma)\}, \{w_e\cobar e\in E(\Gamma)\})$ be a graph of free groups with cyclic edge groups. Choose a graph $X_v$ such that $\pi_1(X_v)=G_v$ and an immersion $\gamma_e\co S^1\to X_{\iota(e)}$ such that $[\gamma_e]=w_e\in G_{\iota(e)}$. For each pair $\{e,\bar e\}$, attach $S^1\times[-1,1]$ to $\coprod_{v} X_v$ such that $S^1\times\{-1\}$ and $S^1\times \{1\}$ are identified with $\gamma_e$ and $\gamma_{\bar e}$, respectively; we denote by $X(\GG)$ the space thus obtained. Moreover, we let $\pi_1(\GG)$ denote the fundamental group of $X(\GG)$. Note that $X(\GG)$ is well-defined up to homotopy equivalence, and so $\pi_1(\GG)$ is well-defined. We call each $X_v$ a \textit{vertex space} of $X(\GG)$, and the closure of each connected component of $X(\GG)\smallsetminus \coprod_v X_v$ a \textit{cylinder} in $X(\GG)$. The image of each $\gamma_e$ is called a \textit{boundary component of the cylinder corresponding to $\{e,\bar{e}\}$}. Let $K\subseteq X(\GG)\smallsetminus \coprod_v X_v$ be the union of disjoint open annular neighborhoods of the cores of all the cylinders in $X(\GG)$. If $C$ is a cylinder in $X(\GG)$ and $H$ is a component of $C\smallsetminus K$ intersecting with $X_v$, then we call $H$ a \textit{half-cylinder incident at $X_v$}.
\ee\label{defn:parameter}
\end{defn}

Note that $X(\GG)$ is an Eilenberg--Mac~Lane space for $\pi_1(\GG)$. The following lemma shows that any closed surface homeomorphically embedded in $X(\GG)$ is $\pi_1$-injective.

\begin{lem}\label{lem:embedded cylinders}
	Let $\GG$ be a graph of free groups with cyclic edge groups and let $S$ be a closed surface homeomorphically embedded into $X(\GG)$. Then:
\be
	\item	$S$ is a union of cylinders;
	\item	the inclusion $S\hookrightarrow X(\GG)$ induces a monomorphism $\pi_1(S)\hookrightarrow\pi_1(\GG)$;
	\item 	$S$ is of non-positive Euler characteristic.
\ee
\end{lem}
\begin{proof}
(1) Suppose there exists $s\in S$ such that for any cylinder $C$ containing $s$, $S$ does not intersect $\int(C)$. Let $D\approx D^2$ be a Euclidean neighborhood of $s$ in $S$ which is so small that $D$ intersects only the cylinders that contains $s$. Then $D\subseteq S$ does not intersect the interiors of any cylinders, by the assumption on $s$. It follows that $D\subseteq X(\GG)^{(1)}$, which is a contradiction.

Hence, for any $s\in S$ there exists a cylinder $C$ such that $s\in C$ and $S\cap\int(C)\ne\varnothing$. In other words, $S\subseteq \cup_{S\cap\int(C)\ne\varnothing} C$. Fix any cylinder $C$ such that $S\cap\int(C)\ne\varnothing$. Since $\int(C)$ is an open cylinder, any sufficiently small regular neighborhood $V\subseteq X(\GG)$ of $t\in S\cap\int(C)$ is actually Euclidean and so, $V\subseteq S$. This means $S\cap\int(C)$ is open in $\int(C)$. On the other hand, compactness of $S$ implies $S\cap\int(C)$ is a closed subset of $\int(C)$. So, $S\cap\int(C)=\int(C)$. From $\cup_{S\cap\int(C)\ne\varnothing} \int(C)\subseteq S\subseteq\cup_{S\cap\int(C)\ne\varnothing} C$, we have $S = \cup_{S\cap\int(C)\ne\varnothing} C$.

(2)	Let $X_0$ be the union of all the vertex spaces of $X(\GG)$ and let $S_0=S\cap X_0$.  Let $\Gamma$ be the underlying graph of $\GG$.  Consider $S$ as a graph of spaces where each connected component of $S_0$ is a vertex space, and each cylinder is an edge space, such that the underlying graph is a subgraph of $\Gamma$.  This gives $\pi_1(S)$ the structure of a graph of groups $\KK$.  Any element of $\pi_1(S)$ that is in reduced form as an element of $\pi_1(\GG)$ is also in reduced form with respect to $\KK$.  Hence the natural map $\pi_1(S)\to\pi_1(\GG)$ is injective.

(3) This is immediate from part (1).
\end{proof}

\begin{rem}\label{rem:embedded cylinders}
In the proof of Lemma~\ref{lem:embedded cylinders} (2), we have only used the assumption that $S$ is a connected union of cylinders, rather than a closed surface.
\end{rem}
As mentioned in the introduction, we are most concerned with the \emph{double} of a free group along a word: this is the a special case of a graph of free groups with cyclic edge groups $\GG=(\Gamma,\{G_v\},\{w_e\})$ in which the underlying graph has two vertices and  a single edge $\{e,\bar{e}\}$, and there is an isomorphism from $G_{\iota(e)}$ to $G_{\iota(\bar e)}$ mapping $w_e$ to $w_{\bar e}$.  We write $X(w) = X(\GG)$.

Let $\gamma_w\co S^1\to\cay(F)/F$ be an immersed loop reading $w$.  By convention, we will always choose $X(w)$ to be the 2--dimensional square complex obtained by taking two copies of $\cay(F)/F$ and gluing a cylinder between the two copies of $\gamma_w\subseteq\cay(F)/F$. 
Given a word $w\in F$ and a finite-index subgroup $F'\le F$,
there is a particular finite cover $Y(w,F')$ of $X(w)$, considered in~\cite{Gordon:2009p360}.
That is, consider the \textit{pullback} $\tilde{\gamma}_w\co T\to S^1$ of $\gamma_w$ along the covering map $\cay(F)/F' \to\cay(F)/F$, yielding the commutative diagram below.  The space $T$ is a finite-sheeted covering space of $S^1$, and hence a disjoint union of circles.  
Then, $Y(w,F')$ is defined to be the graph of spaces corresponding to the first row of the diagram. Note that $X(w)$ is the graph of spaces corresponding to the second row.  
There is a natural covering map $Y(w,F')\to X(w)$ of degree $[F:F']$.
	\[
			\xymatrix{
			\cay(F)/F' \ar[d] && T=\coprod S^1\ar[ll]_{\tilde{\gamma}_w}\ar[rr]^{\tilde{\gamma}_w}\ar[d] && \cay(F)/F' \ar[d]\\
			\cay(F)/F &&  S^1\ar[ll]_{\gamma_w}\ar[rr]^{\gamma_w} && \cay(F)/F.
			}
	\]
A restriction of $\tilde{\gamma}_w$ to a connected component of $T$ is called an \textit{elevation} of $\gamma_w$.  Elevations can be interpreted algebraically:

\begin{lem}[\cite{wilton}, Lemma 2.7]
There is a natural bijection between the set of elevations of $\gamma_w$ and the set of double cosets
\[
F'\backslash F/\langle w\rangle.
\]
For each double coset $F' g\langle w\rangle$, the corresponding elevation $\tilde{\gamma}_w$ is freely homotopic to the lifting of $\gamma_g \gamma_w^{n_g}\gamma_g^{-1}$ with respect to $\cay(F)/F' \to\cay(F)/F$, where $n_g$ is the minimal positive integer such that $g w^{n_g}g^{-1}\in F'$.
\end{lem}

\begin{lem}\label{lem:virtual to double}
Let $w$ be a cyclically reduced word in $F$. 
If $X(w)$ virtually contains a closed hyperbolic surface,
then $Y(w,F')$ contains a homeomorphically embedded closed hyperbolic surface for some $[F:F']<\infty$.
\end{lem}
\begin{proof}
Suppose $X(\GG)$ is a cover of $X(w)$ containing a homeomorphically embedded closed hyperbolic surface $S$ where $\GG$ is a graph of free groups with cyclic edge groups.
By Lemma~\ref{lem:embedded cylinders}, $S$ is a union of cylinders in  $X(\GG)$.  
Let $K$ denote the union of disjoint open annular neighborhoods of the cores of the cylinders in $X(\GG)$ and let $S_0$ be a connected component of $S\smallsetminus K$ such that $\chi(S_0)< 0$.  
Then $S_0$ is contained in the neighborhood $N$ of some vertex space $X'$ of $X(\GG)$.  
Since $X'\to\cay(F)/F$ is a finite cover, $X' = \cay(F)/F' $ for some $[F:F']<\infty$. 
Let $\gamma_w\co S^1\to \cay(F)/F$ denote the immersed loop reading $w$.
Recall that $Y(w,F')$ is obtained by taking two copies of $X'=\cay(F)/F' $ and gluing cylinders along the copies of the connected components of the pullback $\tilde{\gamma}_w$ of $\gamma_w$ along $\cay(F)/F' \to \cay(F)$.
As $S_0$ is the union of some half-cylinders incident at $X'$, which are glued along some elevations of $\gamma_w$, the double $S$ of $S_0$ glued along $\partial S_0$ is a closed hyperbolic surface in $Y(w,F')$.
\end{proof}

A \textit{polygonal disk} $P$ is a 2--dimensional disk such that $\partial P$ comes with the CW-structure of a polygon. If the edges (vertices, respectively) of $P$ are denoted as $x_1,x_2,\ldots, x_m$, we will consider the indices as taken to be modulo $m$, and will implicitly assume that $x_i$ and $x_{i+1}$ intersect with a common vertex (edge, respectively). We let $E(\partial P)$ denote the set of 1-cells (edges) in $\partial P$. A \textit{pairing} on a set $X$ is an equivalence relation such that each equivalence class consists of exactly two elements. A \textit{partial side-pairing} on a collection of polygonal disks $P_1,\ldots,P_m$ is a pairing $\sim$ on a subset of $\coprod_i E(\partial P_i)$ along with a choice of a homeomorphism identifying the edges related by $\sim$.
A \textit{side-pairing} refers to a partial side-pairing which is defined on the whole set $\coprod_i E(\partial P_i)$. For a (partial, respectively) side-pairing $\sim$ on $P_1,\ldots,P_m$, we denote by $\coprod_i P_i/\!\!\sim$ the closed (bounded, respectively) surface obtained from $\coprod_i P_i$ by identifying edges in $\coprod_i \partial P_i$ through $\sim$. For $S=\coprod_i P_i$, note that there is a natural cellular map $P_i\to S$ so that $S$ has the structure of a $2$--dimensional CW--complex. We will let $m(S)$ denote the number of 2-cells in $S$. 

\begin{defn}
Let $w$ be a cyclically reduced word in $F$.
\be
\item
Suppose there exist a partial side-pairing  $\sim$ on some polygonal disks $P_1,\ldots,P_m$, and an immersion $\phi\co S^{(1)}\to \cay(F)/F$ where $S = \coprod_i P_i/\!\!\sim$ such that 
the composition $\partial P_i\to S^{(1)}{\to}\cay(F)/F$ 
is an immersion reading
a non-zero power of $w$ for each $i$. 
Then we say that $S$ is a \textit{$w$--polygonal surface (with respect to $\phi$)}.
\item
We say that $w$ is \textit{polygonal}, if 
there exists a closed $w$--polygonal surface $S$ such that $\chi(S)< m(S)$.
\ee\label{defn:polygonal}
\end{defn}

The most convenient way to specify $\phi$ is to label and orient the edges of the boundaries $\partial P_i$ according to their images under $\phi$.  Let $P$ be a polygonal disk equipped with an orientation and a label by $\{a_1,\ldots,a_n\}$ for each edge of $\partial P$. We say that \textit{$\partial P$ reads $w$} (or, simply \textit{$P$ reads $w$}) if the orientation and the label preserving 
graph map 
$\partial P\to\cay(F)/F$ reads $w$. If $P$ is a polygonal disk whose boundary reads $w$ and $u$ is a subword of $w$, the interval in $\partial P$ reading $u$ will be simply referred to as $u$ when there is no danger of confusion. 
We will also view a $w$--polygonal surface $S$ as obtained by a side-pairing $\sim$ on polygonal disks $P_1,\ldots,P_m$ such that each $\partial P_i$ reads a power of $w$ and $\sim$ respects orientations and labels of the edges; moreover, one requires that at each vertex of $S$ there are no two incoming edges or two outgoing edges of the same label, and $\chi(S)< m(S)$.

\begin{rem}
Let $w\in F$ be a cyclically reduced word which is not a proper power, and $S$ be a closed connected $w$--polygonal surface.
Note that the condition $\chi(S)< m(S)$ is equivalent to forbidding the following three cases:
\enumir
\be
\item
$m(S)=1$ and $S\approx S^2$,
\item
$m(S)=2$ and $S\approx S^2$, or
\item
$m(S)=1$ and $S\approx \R P^2$.
\ee
\enumia
The case (i) occurs when $S = (P/\!\!\sim)\approx S^2$ for some side-pairing $\sim$ on a polygonal disk $P$.
Then there must exist two consecutive edges of $\partial P$ which are identified by $\sim$ such that $\sim$ fixes a common vertex.
This is impossible by the definition of
an immersion.
The case (ii) is equivalent to the condition that $S$ consists of two identical polygonal disks whose boundaries (via immersions into $\cay(F)/F$) both read the same power $w^i$, and $\sim$ identifies each pair of edges corresponding to the same letter of $w^i$.
The case (iii) amounts to saying that $S$ consists of a polygonal disk $P$ such that an immersion $\phi\co\partial P\to\cay(F)/F$ reads $w^k = u^2$ for some $k>0$ and $u\in F$, and $\sim$ identifies each pair of edges on $\partial P$ by a $\pi$--rotation; this means,  the $i$-th letter of $u^2$ is identified with the $(i+|u|)$-th letter of $u^2$. For example, take $k=2$ and $u=w$.
As (i), (ii) and (iii) occur only in these obvious cases, we will often omit the proof that $\chi(S)< m(S)$ for a $w$--polygonal surface $S$ when establishing polygonality of $w$.
\label{rem:polygonal}
\end{rem}

\begin{exmp}
	If $w\in F$ is a surface group relator such as $\prod_i [a_{2i-1},a_{2i}]$ or $\prod_i a_i^2$, then $w$ is polygonal. This is because one can take a polygonal disk whose boundary reads $w$ and glue the edges of the same labels respecting the orientations.
\end{exmp}

\begin{lem}\label{lem:mchi}
Suppose $w\in F$ is cyclically reduced and $S=\coprod_i P_i/\!\!\sim$ is a closed connected $w$--polygonal surface such that $\chi(S) < m(S)$. Suppose $Q$ and $Q'$ are distinct vertices of $\partial P_i$ for some $i$ and one of the two intervals on $\partial P_i$ cut by $\{Q,Q'\}$ reads a power of $w$. Then $Q$ is not identified to $Q'$ by $\sim$.
\end{lem}

\begin{proof}
Let $\phi\co S^{(1)}\to\cay(F)/F$ be the immersion with respect to which $S$ is $w$--polygonal. Name the vertices of $P_i$ as $v_1,\ldots,v_{kl}$ where $k>0$ and $l=|w|$. We may assume $Q=v_1$ and $Q'=v_{hl+1}$ for some $0<h<k$. Suppose $v_1\sim v_{hl+1}$. Since the orientations and the labels of the edges $(v_1,v_2),(v_2,v_3),\ldots$ coincide with those of $(v_{hl+1},v_{hl+2}),(v_{hl+2},v_{hl+3}),\ldots$ respectively, the immersion condition of $\partial P_i\to S^{(1)} \to \cay(F)/F$ implies that $v_2\sim v_{hl+2},v_3\sim v_{hl+3},\ldots$. This is possible only when $\sim$ pairs all the edges of $\partial P_i$ by a $\pi$-rotation; hence, $m(S)=1=\chi(S)$.
\end{proof}

A graph $\Gamma$ embedded in a compact surface with boundary $S'$ is a \emph{spine} if the closure of every connected component of $S'\smallsetminus \Gamma$ is an annulus, with one boundary component contained in $\Gamma$ and one boundary component contained in $\partial S'$.

\begin{lem}\label{lem:polygonal}
Let $w$ be a cyclically reduced word in $F$.
Then $w$ is polygonal if and only if 
there exists a commutative diagram
\[
\xymatrix{
T=\coprod S^1\ar[rrr]^{\tilde{\gamma}_w}\ar[dd]_{\mathrm{}} &&& \cay(F)/F' \ar[dd] \\
& \partial S'\phantom{|}\ar[ul]_<<<<<<{\mathrm{incl.}}\ar@{^{(}->}[r]\ar[dl]_{\mathrm{}}
 & S'\ar[ur]^<<<<<<<{\tilde{\psi}}\ar[dr]& \\
S^1\ar[rrr]^{\gamma_w} &&& \cay(F)/F
}
\]
such that $[F:F']<\infty$, $\gamma_w$ is an immersion reading $w$, $\tilde{\gamma}_w$ is the pullback of $\gamma_w$ along $\cay(F)/F' \to\cay(F)/F$, $S'$ is a hyperbolic surface with boundary, 
$\tilde{\psi}$ is an embedding when restricted to a spine of $S'$ and $\partial S'\to T=\coprod S^1$ is an inclusion.
\end{lem}

\begin{proof}
($\Rightarrow$) 
Choose a closed connected $w$--polygonal surface $S=\coprod P_i/\!\!\sim$ with respect to an immersion $\phi\co S^{(1)}\to\cay(F)/F$ such that $\chi(S)<m(S)$. Since $F$ is subgroup separable, the immersion $\phi$ lifts to an embedding $\tilde{\phi}\co S^{(1)}\to\cay(F)/F' $ for some $[F:F']<\infty$~\cite{Stallings:1983p596}. Let $S'$ be obtained by drilling a hole in the interior of each of  $P_i$'s. Then $\chi(S') = \chi(S) - m(S) < 0$ and $S^{(1)}$ is a spine of $S'$. Define $\tilde{\psi}$ to be the composition $S'\stackrel{\simeq}{\longrightarrow} S^{(1)}{\longrightarrow}\cay(F)/F' $ so that $\tilde{\psi}_*\co \pi_1(S')\to F'$ is injective. Because the boundary of each $P_i$ reads a power of $w$, there is a covering map $\partial S'\to S^1$ whose composition with $\gamma_w$ equals the composition $\partial S'\hookrightarrow S'\to\cay(F)/F' \to\cay(F)/F$.

Let $T=\coprod S^1$ be the pullback of $\gamma_w$ along the covering map $\cay(F)/F' \to\cay(F)/F$.  
By the universal property, the covering map $\partial S'\to S^1$ lifts to $q:\partial S'\to T$.  Let $\partial_i S'$ denote the boundary component of $S'$ drilled on $P_i$. The restriction of $q$ to $\partial_i S'$ is a finite covering map onto its image, say $A\subseteq T$. $A\to S^1\stackrel{\gamma_w}{\to} \cay(F)/F$ reads $w^r$ for some $r>0$. If $\partial_i S'\to A$ is a $k$-to-$1$ map for some $k>1$, then $\partial P_i\to\cay(F)/F$ will read $(w^r)^k$. In particular, $\partial P_i\to S^{(1)}\to\cay(F)/F' $ maps a proper subinterval on $\partial P_i$ onto the image of $A\to \cay(F)/F' $. This is a contradiction to Lemma~\ref{lem:mchi}, for $A\to\cay(F)/F' \to\cay(F)/F$ reads $w^r$. Hence the restriction of $q$ to each boundary component of $S'$ is a homeomorphism onto the image. Suppose $\partial_i S'$ and $\partial_j S'$ map onto the same component of $T$ for $i\ne j$.
Then $\partial P_i\to S^{(1)}\hookrightarrow \cay(F)/F' $ and $\partial P_j\to S^{(1)}\hookrightarrow \cay(F)/F' $ have the same image. It follows that $P_i\cup P_j$ maps to a $2$--sphere in $S$. Since we have excluded the case $\chi(S) = m(S) = 2$, we conclude that $q\co\partial S'\to T$ must be an embedding.

($\Leftarrow$) Take $\Gamma$ to be the spine for $S'$ provided by the hypotheses.  For each component $\partial_i S'$ of $\partial S'$, let $A_i$ be the closure of the component of $S'\smallsetminus \Gamma$ that contains $\partial_i S'$.  Glue disks $Q_i$ to $\partial_i S'$, to form a closed surface $S$.   Let $P_i=Q_i\cup A_i$, a disk with boundary contained in $\Gamma$; give $S$ the CW-structure in which $S^{(1)}=\Gamma$ and the $P_i$ are the 2-cells.  Because $\partial Q_i=\partial_iS'$ reads a power of $w$ and $A_i$ is an annulus, it follows that each $\partial P_i$ also reads a power of $w$ under the composition $\Gamma\to\cay(F)/F' \to \cay(F)/F$.   Note that $\chi(S)-m(S) = \chi(S')<0$, where $m(S)$ is the number of 2-cells $P_i$.\end{proof}

\begin{thm}\label{thm:polygonal}
Let $w\in F$ be a cyclically reduced word.
\be
\item
If $w\in F$ is polygonal, then $X(w)$ virtually contains a closed hyperbolic surface.
\item
If a closed hyperbolic surface embeds into a covering space of $X(w)$ with the finitely generated fundamental group, then $w$ is polygonal.
\ee
\end{thm}

\begin{proof}
	Let $\gamma_w\co S^1\to \cay(F)/F$ be an immersion reading $w$.

(1)	
There exist a hyperbolic surface $S'$ and a finite-index subgroup $F'\le F$ satisfying the commutative diagram in Lemma~\ref{lem:polygonal}, where
$\tilde{\gamma}_w\co T=\coprod S^1\to\cay(F)/F' $ denotes the pullback of $\gamma_w$ along $\cay(F)/F' \to\cay(F)/F$.
Let $S_0$ be the closed surface obtained by taking the double of $S'$ along $\partial S'$.
The composition $\partial S'\hookrightarrow T=\coprod S^1\to\cay(F)/F' $ restricts to distinct elevations on the components of $\partial S'$.
By the definition of $Y(w,F')$, $S_0$ is homeomorphic to a union of distinct cylinders in $Y(w,F')$. 
Note  $\chi(S_0) = 2\chi(S') < 0$.

 (2) 
Since $D(w)$ is the fundamental group of a graph of free groups with cyclic edge groups which is word-hyperbolic, $D(w)$ is subgroup separable~\cite{Wise:2000p790}. If a closed hyperbolic surface $S_0$ homeomorphically embeds into a covering space of $X(w)$ with finitely generated fundamental group, then $S_0$ homeomorphically embeds also into a finite cover of $X(w)$~\cite{Scott:1978p142}. By Lemma~\ref{lem:virtual to double}, $S_0$  may be assumed to homeomorphically embed into $Y(w,F')$ for some $[F:F']<\infty$. By Lemma~\ref{lem:embedded cylinders}, $S_0$ is a union of cylinders. Let $m$ denote the number of cylinders in $S$. Pick one of the two vertex spaces in $Y(w,F')$ and name it simply as $\cay(F)/F' $. Let $K$ be the disjoint union of the open annular neighborhoods of the cores of the cylinders in $Y(w,F')$. Define $Y'$ ($S'$, respectively) to be the connected component of $Y\smallsetminus K$ ($S_0\smallsetminus K$, respectively) intersecting with $\cay(F)/F' $. Note that $\Gamma=S'\cap \cay(F)/F'$ is a spine for $S'$.  Let $\tilde{\gamma}_w\co T=\coprod S^1\to \cay(F)/F' $ denote the pullback of $\gamma_w$ along $\cay(F)/F' \to\cay(F)/F$. Then $Y'$ is the union of $\cay(F)/F' $ and the half-cylinders glued along the restrictions of $\tilde{\gamma}_w$ to the components of $T$. Moreover, $S'$ is some collection of those half-cylinders and so, there exist an inclusion $\partial S'\hookrightarrow T$ yielding the commutative diagram in Lemma~\ref{lem:polygonal}. 
$S'$ is hyperbolic since $\chi(S')=\frac12 \chi(S_0)< 0$.
\end{proof}

\begin{rem}\label{rem:genus}
Let $S$ be a closed connected $w$--polygonal surface. From the proof of Theorem~\ref{thm:polygonal}, one can actually find the degree of a cover $Y(w,F')\to X(w)$ where $Y(w,F')$ contains an embedded closed hyperbolic surface $S_0$. Namely, since the immersion $S^{(1)}\to\cay(F)/F$ lifts to an embedding $S^{(1)}\to \cay(F)/F' $ where the embedding does not increase the number of vertices~\cite{Stallings:1983p596},
\[ \deg(Y(w,F')\to X(w)) = \deg(\cay(F)/F' \to\cay(F)/F)=|(\cay(F)/F')^{(0)}| = |S^{(0)}|.\]
Moreover, one can specify the Euler characteristic of $S_0$ as $\chi(S_0) = 2(\chi(S)-m(S))$.
\end{rem}

\begin{exmp}\label{exmp:q23}
In an earlier version of \cite{Gordon:2009p360}, it was asked whether $D(w)$ contains a surface group for $w=a^{-2} b^{-1}a^{-1}b a b^{-1}ab\in F_2$. Figure~\ref{fig:bs} describes a side-pairing $\sim$ of a polygonal disk $P$ with its boundary reading $w^2$ so that $S=P/\!\!\sim$ is a closed surface of Euler characteristic $-1$. Note that no vertex of $S$ has two incoming or two outgoing edges of the same label. By Remark~\ref{rem:genus}, one can find a closed surface $S'$ in $Y(w,F')$ for some $[F:F'] = |S^{(0)}| = 7 $ such that $\chi(S')= 2(\chi(S)-m(S))=2(-1-1)=-4$. In particular, $D(w)$ contains a surface group. One can similarly show that $w=a^p (a^{-1})^b a^q a^b$ is polygonal for any $pq\ne0$.
\end{exmp}

\begin{figure}[htb!]
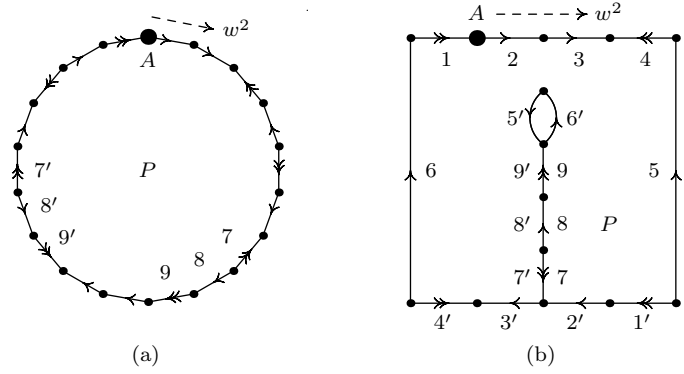

\centering
\subfigure[]{
\includegraphics[]{fig/figbsa.0}
}
\hspace{.3in}
\subfigure[]{
\includegraphics[]{fig/figbsb.0}
}
\caption{A side-pairing on a polygonal disk $P$. Single and double arrows denote $a$-edges and $b$-edges, respectively. 
In (a), starting from the vertex marked by $A$, following the dashed arrow will read $w^2=
(a^2 (a^{-1})^b a a^b)^2$. Let $\sim$ glue a pair of intervals, marked as $\{7,8,9\}$ and $\{7',8',9'\}$, so that $P$ becomes an annulus as shown in (b).
In (b), let a side-pairing $\sim$ further pair the edges marked by $1,\ldots,6$ with the edges marked by $1',\ldots,6'$, respectively. Now $S=P/\!\!\sim$ is a closed $w$--polygonal surface of Euler characteristic $-1$.\label{fig:bs}}
\end{figure}

\begin{cor}\label{cor:polygonal}
	If $w\in F$ is equivalent to a polygonal word, then $D(w)$ contains a surface group. 
\end{cor}
\begin{proof}
We may assume $w$ is polygonal.  By Theorem~\ref{thm:polygonal}, $X(w)$ virtually contains a closed surface $S$. Lemma~\ref{lem:embedded cylinders} implies that $\pi_1(S)$ embeds into $D(w)=\pi_1(X(w))$.
\end{proof}

We note that a polygonal word is not an element of a basis; for otherwise, $D(w)$ would be free.

\begin{conj}[Tiling Conjecture]\label{conj:tiling}
A minimal diskbusting word is polygonal.
\end{conj}

\section{Non-polygonal words}

The following example shows that the hypothesis of minimality is necessary in the Tiling Conjecture.

\begin{exmp}\label{example:nonpolygonal} 
There are diskbusting words which are not polygonal. For example, one can show that $w=abab^2ab^3$ is not polygonal as follows. Suppose there exists a $w$--polygonal surface $S$ such that $\chi(S) < m(S)$. Then there exists a vertex $v$ in $S^{(0)}$ such that $\deg(v)$ is at least 3. From the fact that $w$ is positive, it follows that $\deg(v)$ is even. So $\deg(v)=4$, and in particular, there exists an incoming $a$-edge at $v$. Since the letter $a$ is always followed by $b$ in $w$ and exactly one $b$-edge is outgoing at $v$, $\deg(v) = 2$; hence, we have a contradiction. Note that by the automorphism $a\mapsto ab^{-2}$ and $b\mapsto b$, $w$ maps to $a(a^2)^b$. In Example~\ref{exmp:bs2}, $a(a^2)^b$ is shown to be polygonal. It is not hard to see that $a(a^2)^b$ is minimal. 
\end{exmp}

So we see that a double may contain a surface group which does not manifest itself as an embedded surface in a cover.  One can interpret this in terms of the existence of compact complexes which are homotopic to a surface but which do not contain embedded surfaces, as the next proposition shows.  We refer the reader to \cite{BridsonHaefliger} for the definition and properties of CAT(0) spaces.  A square complex is non-positively curved (meaning that its universal cover is CAT(0)) if and only if the link of every vertex contains no cycles of length less than four, by Gromov's Link Condition.  In particular, the complex $X(w)$ is a non-positively curved square complex.

\begin{prop}\label{prop:exotic surface}
There exists a compact, non-positively curved square complex $Y$ such that $\pi_1(Y)$ is a hyperbolic surface group but such that no subspace of $Y$ is homeomorphic to a closed surface.
\end{prop}
\begin{proof}
Let $w$ be a diskbusting, non-polygonal word that is equivalent to a polygonal one, as in Example \ref{example:nonpolygonal}.  Then $D(w)$ contains a surface subgroup $H$.   This subgroup is quasiconvex (this follows, for instance, from the fact that $H$ is a virtual retract \cite{wilton}).  The universal cover $\widetilde{X}$ of $X(w)$ is a Gromov-hyperbolic CAT(0) square complex, quasi-isometric to the Cayley graph of $D(w)$ by the \v{S}varc--Milnor Lemma.  Because every quasigeodesic in a Gromov-hyperbolic metric space is uniformly close to a geodesic, every orbit of $H$ is quasiconvex in the 1-skeleton of $\widetilde{X}$. Therefore, by \cite[Theorem H]{Haglund}, it follows that $H$ acts cocompactly on a convex subcomplex $\widetilde{Y}$ of $\widetilde{X}$.

Let $Y=\widetilde{Y}/H$.  This is a compact, non-positively curved square complex with fundamental group $H$.  However, the inclusion $\widetilde{Y}\hookrightarrow\widetilde{X}$ descends to an embedding $Y\hookrightarrow\widetilde{X}/H$, a covering space of $X(w)$ with fundamental group $H$.  Therefore, by Theorem \ref{thm:polygonal}, no compact subspace of $Y$ is homeomorphic to a compact surface.
\end{proof}

\begin{rem}
As this article was in its final stages of revision, we heard of a simpler construction of a similar example, due to Noel Brady \cite{mccammond:personal}.
\end{rem}

\section{Words Without Isolated Generators}\label{sec:no isolated}
We consider words of the following form in $F_n = \langle a_1,\ldots,a_n\rangle$:
\[
w = \prod_{i=1}^l a_{k_i}^{p_i}, \qquad\textrm{where }  
l>0, k_i\ne k_{i+1} \textrm{ and } |p_i|>1.
\]
Note that we take the index $i$ modulo $l$. Such $w$ is said to have \textit{no isolated generators}.

Throughout this section, we will employ the following notations. Take two identical polygonal disks $P$ and $P'$ such that $\partial P$ and $\partial P'$ read $w$. Put $p = \sum_{i=1}^l |p_i|$. Label the edges of the $p$-gon $\partial P$ as $e_1,e_2,\ldots,e_p$. Fix an orientation-preserving homeomorphism $h\co P\to P'$ such that $e_j$ and $h(e_j)$ have the same label and the orientation. Consider a partial side-pairing $\sim$ on $P\coprod P'$ such that $e_j\sim h(e_{j+1})$ if and only if the label and the orientation of $e_j$ coincide with those of $h(e_{j+1})$; see Figure~\ref{fig:no isolated} (a). For each $1\le i\le l$, there exists a boundary component $\partial_i$ of $S=P\coprod P'/\!\!\sim$ such that $\partial_i$ is a $2$-cycle whose edges are labeled by $a_{k_i}$ and $a_{k_{i+1}}$ (Figure~\ref{fig:no isolated} (b)).

Let $\rho_i = (\sign(p_i){k_i},\sign(p_{i+1}){k_{i+1}})$ for each $1\le i\le l$.   Recall from Section~\ref{sec:graph of graphs} that an edge of $S^{(1)}$ which is labeled by $a_k$ is called an \textit{$a_k$-edge}.   If $\sign(p_i)=1$ then there exists an incoming $a_{k_i}$-edge at $\partial_i$, and if $\sign(p_{i+1})=-1$ then there exists an incoming $a_{k_{i+1}}$-edge at $\partial_i$.   Likewise, if $\sign(p_i)=-1$ then there exists an outgoing $a_{k_i}$-edge at $\partial_i$, and if $\sign(p_{i+1})=1$ then there exists an outgoing $a_{k_{i+1}}$-edge at $\partial_i$.

Let $V_n = \{\pm1,\pm2,\ldots,\pm n\}$ and let
\[
E_n = \{(i,j)\cobar i,j\in V_n\textrm{ and } |i|\ne|j|\}\subseteq V_n\times V_n.
\]
Define $\AA^*_n$ to be the free abelian monoid on the set $E_n$, and let $\AA_n$ be the quotient of $\AA^*_n$ by the relation $(i,j)=(-j,-i)$. Then $\AA_n$ is also a free abelian monoid. Put $\rho(w) = \sum_{i=1}^l \rho_i\in\AA_n$.
Define $\TT_n\le \AA_n$ by 
\[
\TT_n = \langle \{ (c_1,c_2)+(c_2,c_3)+\cdots+(c_r,c_1) \cobar c_1,c_2,\ldots,c_r\in V_n\textrm{ are all distinct}\}\rangle.
\]

\begin{figure}[htb!]
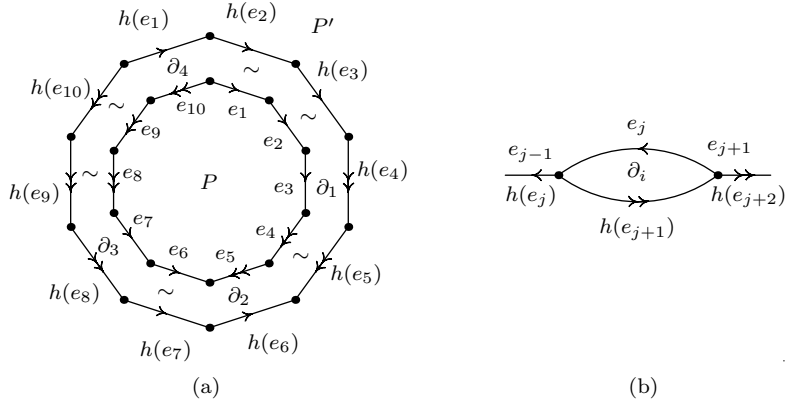

	\centering
	\subfigure[]{
	\includegraphics[]{fig/fig10gon1.0}
	}
\hspace{.3in}
	\subfigure[]{
	\includegraphics[]{fig/fig10gonpartial.0}
	}
	\caption{An example of $S$ when $w=a_1^3a_2^2 a_1^{-2}  a_2^{-3}$. 
	Single and double arrows denote $a_1$-edges and $a_2$-edges, respectively.
	(a) Rotate $\partial P'$  by $\frac{2\pi}{10}$ and identify the overlapping edges if the labels and the orientations coincide. 
	(b) Edges intersecting with a boundary component $\partial_j$ 
	of $S=P\coprod P'/\!\!\sim$. In this figure, $\rho_j = (-1,2)$.
	\label{fig:no isolated}}
\end{figure}

\begin{thm}\label{thm:no isolated}
	Let $w = \prod_{i=1}^l a_{k_i}^{p_i}\in F_n$ 
	where $l>0, k_i\ne k_{i+1}$ and $|p_i|>1$.
	If $\rho(w)\in \TT_n$, then $w$ is polygonal.
\end{thm}

\begin{proof}
We may assume $l>1$, as proper powers are polygonal by definition.  We will define a side-pairing $\sim'$ on $P\coprod P'$ such that $\sim'$ is an extension of $\sim$.
Put $p=\sum_{i=1}^l |p_i|$.
Since $\rho(w) = \sum_{i=1}^l \rho_i\in\TT_n$, we can write $\{1,2,\ldots,p\} = \coprod_{h=1}^q A_h$ so that for each $h$, $\sum_{i\in A_h} \rho_i$ is a generator of $\TT_n$.
Choose any $1\le h\le q$ and write $A_h = \{i_1,\ldots,i_r\}$.
We will define $\sim'$ for the edges of $\partial_{i_1}\cup\cdots\cup\partial_{i_r}$. We have $\rho_{i_1}+\rho_{i_2}+\cdots+\rho_{i_r} = (c_1,c_2)+(c_2,c_3)+\cdots+(c_r,c_1)$ for some distinct $c_1,c_2,\ldots,c_r\in V_n$. Note that at a vertex of $\partial_{i_j}$, an $a_{|c_j|}$-edge is incoming or outgoing according to whether $\sign(c_j)=1$ or $\sign(c_j)=-1$. Also, an $a_{|c_{j+1}|}$-edge intersects $\partial_{i_j}$ in a similar fashion; see Figure~\ref{fig:no isolated bd}. Let $\sim'$ glue the $a_{|c_j|}$-edge of $\partial_{i_{j-1}}$ to the $a_{|c_j|}$-edge of $\partial_{i_j}$, for $j=1,2,\ldots,r$. At any vertex of $\partial_{i_1}\cup\cdots\cup\partial_{i_r}/\!\!\sim'$, there do not exist two incoming or two outgoing edges of the same label, since $|c_j|$'s are all distinct. As $1\le h\le q$ was arbitrarily chosen, $\sim'$ is defined for all the edges of $\partial S = \partial (P\coprod P'/\!\!\sim)$. We can consider $\sim'$ as an extension of $\sim$ such that $P\coprod P'/\!\!\sim'$ is a closed $w$--polygonal surface.
\end{proof}

\begin{figure}[htb!]
	\includegraphics[]{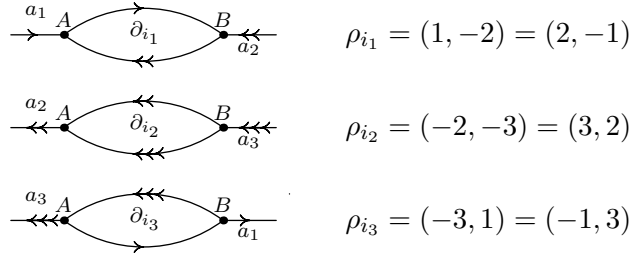}
	\caption{$\rho_{i_1}+\rho_{i_2}+\rho_{i_3}$ is a generator of $\TT_n$. 
	The single, double and the triple arrows denote
	edges labeled by $a_1,a_2$ and $a_3\in F_n$, respectively. Note that
	there exists exactly one incoming $a_i$-edge and one outgoing $a_i$-edge 
	at each vertex of
	$\partial_{i_1}\cup\partial_{i_2}\cup\partial_{i_3}$, for each $i=1,2,3$.
	\label{fig:no isolated bd}}
\end{figure}

\begin{rem}\label{rem:no isolated}
\be
\item
In Theorem~\ref{thm:no isolated}, $\AA_n$ and $\TT_n$ are independent of the choices of $w$. Also,	$\rho(w)$ can be easily computed as
\[\rho(w) = \sum_i (\sign(p_i)a_{k_i},\sign(p_{i+1})a_{k_{i+1}}).\]
\item
In Theorem~\ref{thm:no isolated}, the hypothesis that $\rho(w)\in\TT_n$ is necessary. 
For example, $w=a^2 b^2 c^3 b^{-3}\in\langle a,b,c\rangle=F_3$ 
does not have any isolated generators. But 
\[\langle w, a, b^2 c b^{-2}\rangle
=\langle b^2c^3 b^{-3}, a , b^2 c b^{-2}\rangle = \langle b^{-1}, a, b^2 c b^{-2}\rangle = \langle a,b,c\rangle,
\]
It follows that $w$ is an element of a basis, and hence not polygonal.
\ee
\end{rem}

\begin{cor}
Let $w = \prod_{i=1}^l a_{k_i}^{p_i}\in F_n$ for $l>0, k_i\ne k_{i+1}$ and $|p_i|>1$.
If $k_i$'s are all distinct, then $w$ is polygonal.
\label{cor:no isolated}
\end{cor}

\begin{proof}
	Since $a_{k_i}$'s are all distinct, $\rho(w)=\sum_i 
	(\sign(p_i){k_i},\sign(p_{i+1}){k_{i+1}})$ is a generator of $\TT_n$.
\end{proof}
	
\begin{exmp}
	\be
	\item
	$w=a_1^2a_2^{-3}a_3^{2}\in F_3$ is polygonal, by Corollary~\ref{cor:no isolated}.
	\item
	Let $w=a_1^6 a_2^{-3} a_3^{5} a_2^4 a_3^{-7}\in F_3$.
	We have
	$\rho(w) = (1,-2) + (-2,3) + (3,2) + (2,-3) + (-3,1) =
				(1,-2) + (-2,3) + (-2,-3) + (3,-2) + (-3,1)
				= ( (1,-2) + (-2,-3) + (-3,1)  )+ ((-2,3) + (3,-2))
				\in\TT_3$. By Theorem~\ref{thm:no isolated}, $w$ is polygonal.
	\ee
\end{exmp}

When $n=2$, we will show that any word without isolated generators is polygonal.
We first need a combinatorial lemma.

\begin{lem}\label{lem:sourcesink}
	Let $C$ be a $2m$-gon for $m>0$, so that each edge comes with an orientation.
	Label the edges of $C$ by \textit{clean} or \textit{dirty}, so that
	clean edges and dirty edges alternate. Name each vertex $v$ as:
	\bi
	\item
	a \textit{source}, if a clean edge and dirty edge are outgoing from $v$,
	\item
	a \textit{sink}, if a clean edge and dirty edge are incoming to $v$,
	\item
	a \textit{filter}, if a dirty edge is incoming to and a clean edge is outgoing from $v$, and
	\item
	a \textit{pollutant}, if a clean edge is incoming to and a dirty edge is outgoing from $v$.
	\ei
	Then the number of sources equals that of sinks. Also, the number of filters equals 
	that of pollutants.
\end{lem}

\begin{proof}
	Label the vertices of $C$ as $v_1,\ldots,v_{2m}$ so that the edge $(v_{2i-1},v_{2i})$
	is clean for each $i$. Give $C$ the orientation following $v_1,v_2,\ldots,v_{2m}$.
	Let $h_i=1$ if the orientation of the edge $(v_i,v_{i+1})$ coincides with that of $C$, and $h_i=-1$ otherwise.
For $i=1,\ldots,2m$, define
\begin{eqnarray*}
\tau(i) &=& \left(\frac{1-(-1)^i}{2},\frac{1+(-1)^i}{2}\right)\in\Z^2\\
\sigma_i&=& h_i\tau(i)-h_{i-1}\tau(i-1)
\end{eqnarray*}
One can check
	\[
	\sigma_i=
	\left\{
	\ba{ll}
		(1,1), &\qquad \textrm{if }v_i\textrm{ is a source};\\
		(-1,-1), &\qquad \textrm{if }v_i\textrm{ is a sink};\\
		(1,-1), &\qquad \textrm{if }v_i\textrm{ is a filter};\\
		(-1,1), &\qquad \textrm{if }v_i\textrm{ is a pollutant}.\\
	\ea
	\right.
\]
For example, if $i$ is odd and $v_i$ is a source
then $\tau(i) = (1,0), \tau(i-1)=(0,1)$
and $h_i=1 = -h_{i-1}$. So, $\sigma_i=(1,0)+(0,1)=(1,1)$; see Figure~\ref{fig:sourcesink}.
Note that $\sum_i \sigma_i =\sum_i h_i\tau(i) - \sum_i h_{i-1} \tau(i-1)  =  \sum_i h_i\tau(i) + \sum_i (-h_i \tau(i)) = (0,0)$.
If we let $x,y,z,w$ be the numbers of sources, sinks, filters and pollutants, respectively,
then $\sum_i \sigma_i = (x-y+z-w, x-y-z+w)=(0,0)$. This implies that $x=y$ and $z=w$.
\end{proof}

\begin{thm}\label{thm:no isolated 2}
	If $w\in F_2$ is a word without isolated generators,
	then  $\rho(w)\in\TT_2$. In particular, $w$ is polygonal.
\end{thm}

\begin{proof}
We let $F_2=\langle a,b\rangle$. 
Let $w = \prod_{i=1}^l a^{p_i}b^{q_i}$ for $l>1,|p_i|>1$ and $|q_i|>1$. Pick two identical polygonal disks $P$ and $P'$ such that $\partial P$ and $\partial P'$ read $w$. We consider the same partial side-pairing $\sim$ that was considered in the proof of Theorem~\ref{thm:no isolated} and put $S=P\coprod P'/\!\!\sim$. Define $C$ to be an oriented $2l$-gon, whose vertices are named as $v_1,\ldots,v_{2l}$. Label the edges $(v_{2i-1},v_{2i})$ as \textit{clean} and the other edges as \textit{dirty}. Also, give $(v_i,v_{i+1})$ the orientation coinciding that of $C$ if $p_i>0$, and the opposite orientation otherwise. Using the terms in Lemma~\ref{lem:sourcesink}, each source $v_i$ of $C$ corresponds to a boundary component $\partial_i$ of $S$ such that there exists an outgoing $a$-edge and an outgoing $b$-edge at the two vertices of $\partial_i$. Following the notation in the proof of Theorem~\ref{thm:no isolated} for $a_1=a$ and $a_2=b$, one can write $\rho(\partial_i) = (-1,2) = (-2,1)\in \AA_2$ (see Figure~\ref{fig:sourcesink}). Similarly, a sink corresponds to a boundary component $\partial_i$ such that $\rho(\partial_i) = (1,-2) = (2,-1)$, a filter to $\rho(\partial_i) = (2,1) = (-1,-2)$, and a pollutant to $\rho(\partial_i) = (1,2) = (-2,-1)$. By Lemma~\ref{lem:sourcesink}, $\rho(w) = \sum_{i=1}^s ( (1,-2)+(-2,1) ) + \sum_{i=1}^f ( (1,2)+(2,1))\in\TT_2$ where $s$ and $f$ are the numbers of sources and filters, respectively. By Theorem~\ref{thm:no isolated}, $w$ is polygonal.
\end{proof}

\begin{figure}[htb!]
	\includegraphics[]{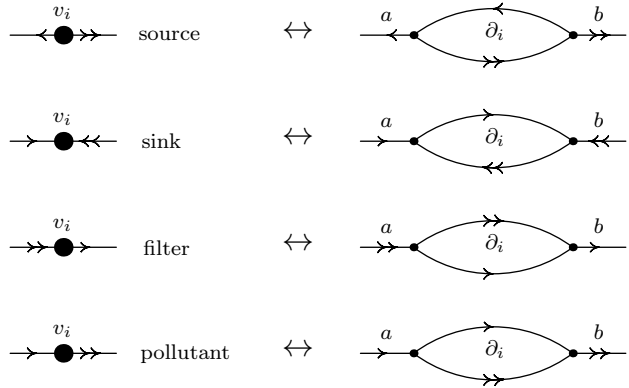} 
	\caption{Proof of Theorem~\ref{thm:no isolated 2}.
Single and double arrows denote $a$-edges and $b$-edges, respectively.
	\label{fig:sourcesink}}
\end{figure}

The following theorem addresses certain cases when all the $b$-letters are isolated, while
none of the $a$-letters are
isolated.

\begin{thm}\label{thm:no isolated 3}
Let $w=\prod_{i=1}^l a^{p_i}b^{q_i}$ such that
$l>1, |p_i|>1, |q_i|=1$,
and $\sum_i (\sign(p_iq_i)+\sign(p_{i+1}q_i))=0$.
Then $w$ is polygonal.
\end{thm}

\begin{proof}
	As in the proof of Theorem~\ref{thm:no isolated 2},
	take two identical polygonal disks $P$ and $P'$ reading $w$, rotate $P'$
	by $2\pi/|w|$, and identify the overlapping edges of $P$ and $P'$ whenever the label
	and the orientation coincide. Denote by $S$ the bounded surface obtained by such a partial side-pairing $\sim$.
	 Let $\partial_i$ be the 
	boundary component of $S$ containing the image of the interval $b^{q_i}$ from $\partial P$.
	Put
	\[ r_i = (\sign(p_i),\sign(q_i),\sign(p_{i+1}))\] for each $i=1,\ldots,l$.
	There are four types of boundary components of $S$ as shown in 
	Figure~\ref{fig:sourcesink2} (compare this to Figure~\ref{fig:sourcesink}).
	This time, we call $\partial_i$ as
	\bi
	\item
	a source if $r_i= (-1,1,1)$ or $(-1,-1,1)$,
	\item
	a sink if  $r_i= (1,1,-1)$ or $(1,-1,-1)$,
	\item
	a filter if  $r_i= (1,1,1)$ or $(-1,-1,-1)$,
	\item
	a pollutant if  $r_i= (-1,1,-1)$ or $(1,-1,1)$.
	\ei
	Let $x,y,z$ and $w$ denote the numbers of sources, sinks, filters and pollutants,
	respectively.
	Then, 
	\begin{eqnarray*}
	x &=& |\{i\cobar p_i p_{i+1}<0\textrm{ and }p_{i}<0\}|
	= \sum_i \left(\frac{1-\sign(p_i p_{i+1})}2 \right)
	 		\left(\frac{1-\sign(p_{i})}2 \right).\\
	y &=&  	|\{i\cobar p_i p_{i+1}<0\textrm{ and }p_{i}>0\}|
			=
			\sum_i \left(\frac{1-\sign(p_i p_{i+1})}2 \right)
	 		\left(\frac{1+\sign(p_{i})}2 \right).\\
	z &=& |\{i\cobar p_i p_{i+1}>0\textrm{ and }p_iq_i>0\}|
	= \sum_i \left(\frac{1+\sign(p_i p_{i+1})}2 \right)
 		\left(\frac{1+\sign(p_iq_i)}2 \right).\\
	w &=& |\{i\cobar p_i p_{i+1}>0\textrm{ and }p_iq_i<0\}|
	= \sum_i \left(\frac{1+\sign(p_i p_{i+1})}2 \right)
 		\left(\frac{1-\sign(p_iq_i)}2 \right).\\
	\end{eqnarray*}
 	Hence,
	\[ 4x-4y = \sum_i (-2\sign(p_i)+2\sign(p_i^2p_{i+1}))
	=2\sum_i (-\sign(p_i) + \sign(p_{i+1})) = 0\] and
	\[4z-4w = \sum_i (2\sign(p_iq_i)+2\sign(p_i^2p_{i+1}q_i))
	=2\sum_i (\sign(p_iq_i) + \sign(p_{i+1}q_i)) = 0.\]
One can identify a source with a sink (a filter with a pollutant, respectively) maintaining that the surface after the identification is still $w$--polygonal; each identification is described in Figure~\ref{fig:sourcesink2}. Hence, there exists an extension $\sim'$ of $\sim$ such that $P\coprod P'/\!\!\sim'$ is a closed $w$--polygonal surface.
\end{proof}
	
\begin{figure}[htb!]
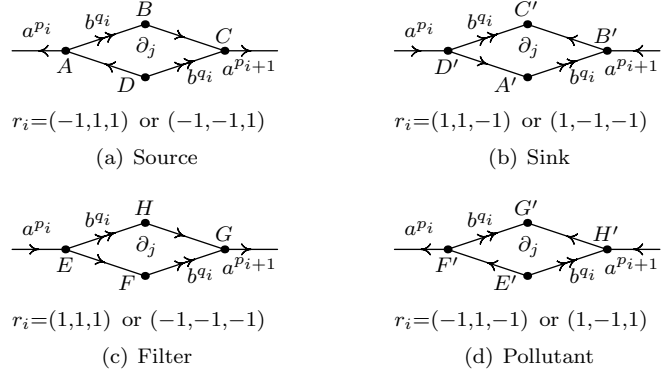

\centering 
\subfigure[Source]{ \includegraphics[]{fig/figsourcesink2.1} } 
\hspace{.4in} 
\subfigure[Sink]{ \includegraphics[]{fig/figsourcesink2.2} }\\ 
\subfigure[Filter]{ \includegraphics[]{fig/figsourcesink2.3} } 
\hspace{.4in} 
\subfigure[Pollutant]{ \includegraphics[]{fig/figsourcesink2.4} } 
\caption{Proof of Theorem~\ref{thm:no isolated 3}. If the source in (a) and the sink in (b) are glued by identifying the edges $AB,BC,CD$ and $DA$ with $A'B',B'C',C'D'$ and $D'A'$, respectively, the resulting surface is still $w$--polygonal. Similarly, the filter in (c) and the pollutant in (d) can be glued by identifying $EF,FG,GH$ and $HE$ with $E'F',F'G',G'H'$ and $H'E'$ respectively, without generating any vertex with two incoming or two outgoing edges of the same label.
\label{fig:sourcesink2}}
\end{figure}

\begin{cor}\label{cor:no isolated 3}
	$\prod_{i=1}^l a^{p_{2i-1}} (a^{p_{2i}})^b$ is polygonal
 	where $l>0$ and $|p_i|>1$ for each $i$.
\end{cor}

\begin{proof}
	Write
	$w = \prod_i a^{p_{2i-1}} b^{q_{2i-1}} a^{p_{2i}} b^{q_{2i}}$
	where $q_{2i-1}=-1$ and $q_{2i}=1$.
	Then, 
	$\sum_i (\sign(p_iq_i)+\sign(p_{i+1}q_i))
	=\sum_i \sign(p_i) (\sign(q_i)+\sign(q_{i-1}))=0$.
	By Theorem~\ref{thm:no isolated 3}, $w$ is polygonal.
\end{proof}

\section{Simple Height-1 Words}\label{sec:simple height 1}
A word $w\in F_2=\langle a,b\rangle$ is a \textit{simple height-one word} if for some $l>0$, $w=\prod_{i=1}^l a^{p_i} (a^{q_i})^b$ where $p_i$'s are of the same sign and $q_i$'s are of the same sign. In particular, if $p_i,q_i>0$, then $w$ is called a \textit{positive height-one} word. Note that the isomorphism $F_2\cong\langle a,a^b\rangle$ defined by $a\mapsto a$ and $b\mapsto a^b$ maps positive words to positive height-one words. In this section, we prove polygonality for a class of simple height-one words, including the Baumslag--Solitar relator $a^p (a^q)^b$ for $pq\ne0$. In some sense that will be made clear in Remark~\ref{rem:simple height 1}, \textit{most} simple height-one words turn out to be polygonal. We will mainly focus on positive height-one words as the other cases can be dealt with in almost identical ways.

\begin{defn}\label{defn:consistent}
Let $w=\prod_{i=1}^l a^{p_i} (a^{q_i})^b$ be a positive height-one word. Define $p_{l+1},p_{l+2},\ldots$ and $q_{l+1},q_{l+2},\ldots$ by requiring $p_{i+l}=p_i$ and $q_{i+l}=q_i$ for $i>0$. Fix $m>0$ and consider polygonal disks $P_1,\ldots,P_m$ whose boundary edges are oriented and labeled by $a$ or $b$ so that $\partial P_i$ reads $w^{k_i} = \prod_{j=1}^{k_i l} a^{p_j} (a^{q_j})^b$ for some $k_i>0$. On $\partial P_i$, we will denote by $\alpha_j^{(i)}$ ($\beta_j^{(i)}$, respectively) the $b$-edge between the intervals $a^{p_j}$ and $a^{q_j}$ ($a^{q_j}$ and $a^{p_{j+1}}$, respectively); see Figure~\ref{fig:consistent} (a). A partial side-pairing $\sim$ on $\coprod_i P_i$ is a \textit{consistent $b$-side-pairing} if $\sim$ pairs all $b$-edges of $\coprod_i \partial P_i$ such that each $\alpha^{(i)}_j$ is identified with some $\beta_{j'}^{(i')}$.
\end{defn}

We let $\BB$ denote the free abelian monoid generated by $\lambda^+_{(c_1,\ldots,c_r)}$ and $\lambda^-_{(c_1,\ldots,c_r)}$ where $(c_1,\ldots,c_r)$ ranges over all finite ordered tuples of positive integers. Let $\sim$ be a consistent $b$-side-pairing on $\coprod_i P_i$ and put $S = \coprod_i P_i/\!\!\sim$. Each boundary component $\partial_i$ of $S$ consists of $a$-edges whose orientations induce the same orientation on $\partial_i\approx S^1$. We will continue to use the notation of Definition~\ref{defn:consistent}. If $\alpha_j^{(i)}$ and $\beta_{j'}^{(i')}$ are paired by $\sim$, then some boundary component $\partial_1\subseteq\partial S$ contains two adjacent intervals $a^{p_j}$ from $\partial P_i$ and $a^{p_{j'+1}}$ from $\partial P_{i'}$ which are concatenated at the terminal vertex of $\alpha_j^{(i)}\sim\beta_{j'}^{(i')}$; see Figure~\ref{fig:consistent}. Similarly, another boundary component $\partial_2\subseteq\partial S$ contains the intervals $a^{q_{j'}}$ from $\partial P_{i'}$ and $a^{q_{j}}$ from $\partial P_i$ which are concatenated at the initial vertex of $\alpha_j^{(i)}\sim\beta_{j'}^{(i')}$. Note that at each boundary component $\partial_i$ of $S$, $b$-edges are all incoming or all outgoing at their intersections with $\partial_i$. Let $c_1,c_2,\ldots,c_r$ denote the lengths of the intervals (in this order) on $\partial_i$, partitioned by the vertices intersecting with $b$-edges. If $b$-edges are all incoming at $\partial_i$, we define $\lambda(\partial_i) = \lambda^-_{(c_1,\ldots,c_r)}\in\BB$, and if $b$-edges are all outgoing at $\partial_i$, define $\lambda(\partial_i) = \lambda^+_{(c_1,\ldots,c_r)}$. For $\partial S = \partial_1 \cup\partial_2\cup\cdots$, we define $\lambda(S) = \sum_i\lambda(\partial_i)$. (This is an abuse of notation, as $\lambda(S)$ depends on the choice of the expression $S = \coprod_i P_i/\!\!\sim$ and the immersion $S^{(1)}\to\cay(F_2)/F_2$.)
\enumia

\begin{figure}[htb!]
	\includegraphics[]{fig/figconsistent.0}
	\caption{
	A consistent $b$-side-pairing on $P_i\coprod P_{i'}$.
	\label{fig:consistent}
	}
\end{figure}

\begin{figure}[htb!]
	\centering
	\subfigure[$P$]{
	\includegraphics[]{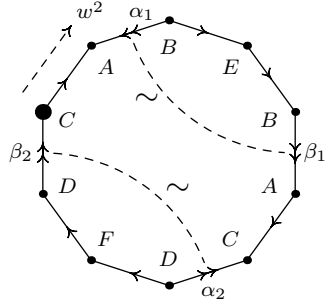}
	}
	\subfigure[$S=P/\!\!\sim$]{
	\includegraphics[]{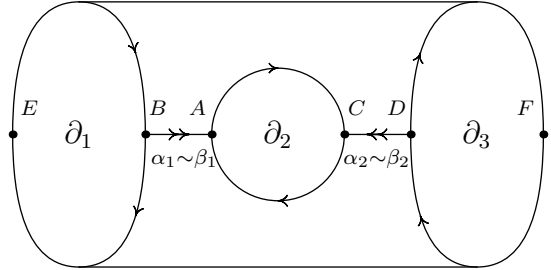}
	}
	\caption{
	(a) A polygonal disk $P$ where $\partial P$ reads $w^2$ for $w=a (a^2)^b$. 
	Dashed curves show the consistent $b$-side-pairing $\sim$.
	(b) $S=P/\!\!\sim$ is a three-punctured sphere.
	\label{fig:partial}
	}
\end{figure}

\begin{exmp}\label{exmp:monoid u}
	Consider the Baumslag--Solitar relator $w=a (a^2)^b$. By the consistent $b$-side-pairing $\sim$ on a polygonal disk $P$ shown in Figure~\ref{fig:partial}, $S=P/\!\!\sim$ is a three-punctured sphere with boundary components $\partial_1,\partial_2$ and $\partial_3$. One sees that $\lambda(\partial_1)=\lambda(\partial_3) = \lambda^+_{(2)}$ and $\lambda(\partial_{2}) = \lambda^-_{(1,1)}$. If one takes an identical copy $S'$ of $S$ with each boundary component $\partial_i'\subseteq\partial S'$ corresponding to $\partial_i\subseteq\partial S$, a closed surface $S_0$ can be formed by identifying $\partial_1$ to $\partial_2$, $\partial_1'$ to $\partial_2'$, and $\partial_3$ to $\partial_3'$. In order to show that $S_0$ is $w$--polygonal, one has only to be careful about the identification of $\partial_3$ with $\partial_3'$; that is, we identify $\partial_3$ to $\partial_3'$ by a $\pi$-rotation so that the vertex on $\partial_3$ to which a $b$-edge is outgoing (the vertex $D$ in the Figure~\ref{fig:partial} (b)) is not identified with its copy on $\partial_3'$. This implies that the Baumslag--Solitar word $w$ is polygonal.
\end{exmp}

\begin{defn}\label{defn:monoid u}
We let $\UU$ be the submonoid of $\BB$ generated by the following elements.
\bi
\item
$\lambda^+_{(c_1,\ldots,c_r)}+ \lambda^-_{(d_1,\ldots,d_s)}$ 
where $\sum_i c_i = \sum_j d_j$.
\item
$\lambda^+_{(c_1,\ldots,c_r)}+ \lambda^+_{(d_1,\ldots,d_s)}$ 
and 
$\lambda^-_{(c_1,\ldots,c_r)}+ \lambda^-_{(d_1,\ldots,d_s)}$,
where
 $m=\sum_i c_i = \sum_j d_j$ and there exists $c$ such that
the sets $\{c+c_1,c+c_1+c_2,\ldots,c+c_1+\cdots+c_r = c+m\}$ and
$\{d_1,d_1+d_2, \ldots,d_1+\cdots+d_s=m\}$
are disjoint modulo $m$.
\ei
\end{defn}

\begin{lem}\label{lem:simple height 1}
	Let $w$ be a simple height-one word. Suppose $P_1,\ldots,P_k$ are polygonal disks reading some powers of $w$, and $\sim_i$ is a consistent $b$-side-pairing on $P_i$ for each $i$. If $\sum_i \lambda(P_i/\!\!\sim_i)\in \UU$, then $w$ is polygonal.
\end{lem}

\begin{proof}
We will only prove the theorem for positive height-one words, as the other cases are very similar. Write $w = \prod_{i=1}^l a^{p_i}(a^{q_i})^b$ for $p_i,q_i>0$. Let $S_i = P_i/\!\!\sim_i$ for each $i$. One can rewrite the sum $\sum_i \lambda(S_i) = \sum_j (\lambda(\partial_j)+\lambda(\partial_j'))$, where $\{\partial_j,\partial_j'\cobar j=1,2,\ldots\}$ is the collection of the boundary components of $\coprod_i S_i$ such that each $\lambda(\partial_j)+\lambda(\partial_j')$ is a generator of $\UU$. In particular, the lengths (that is, the number of edges) of $\partial_j$ and $\partial_j'$ are the same. We will define a pairing $\sim$ of the $a$-edges on $\partial(\coprod_i S_i)$ so that $\partial_j$ and $\partial_j'$ are identified for each $j$. Fix $\partial = \partial_j$ and $\partial'=\partial_j'$.

First, consider the case when $\lambda(\partial)=\lambda^+_{(c_1,\ldots,c_r)}$ and $\lambda(\partial')=\lambda^-_{(d_1,\ldots,d_s)}$ such that $\sum_i c_i = \sum_j d_j$. In this case, we let $\sim$ be any identification of $\partial$ and $\partial'$ respecting the orientations of the $a$-edges. This corresponds to the identification of $\partial_1$ to $\partial_2$ and $\partial_1'$ to $\partial_2'$ in Example~\ref{exmp:monoid u}. Recall that at each vertex of $\partial\subseteq\coprod_i \partial S_i$ ($\partial'$, respectively), either no $b$-edges are intersecting, or exactly one $b$-edge is outgoing (incoming, respectively). Hence, at each vertex in the image of $\partial\sim\partial'$ in $\coprod_i S_i/\!\!\sim$, no two $b$-edges can be both incoming or both outgoing.

Next, suppose $\lambda(\partial)=\lambda^+_{(c_1,\ldots,c_r)}$ and $\lambda(\partial')=\lambda^+_{(d_1,\ldots,d_s)}$; put $m=\sum_i c_i = \sum_j d_j$, and choose $c$ such that $c+\sum_{i=1}^h c_i\not\equiv \sum_{j=1}^{k} d_j \quad(\textrm{mod } m)$ for all $h,k$. Label the vertices of $\partial$ by $v_0,v_1,\ldots,v_{m-1}$ consecutively, such that $v_s$ intersects with an outgoing $b$-edge if and only if $s = \sum_{i=1}^h c_i$ for some $h>0$. Moreover, we choose the labeling so that the orientation of the $a$-edge $(v_i,v_{i+1})$ is from $v_i$ to $v_{i+1}$. Similarly, label the vertices of $\partial'$ by $v_0',v_1',\ldots,v_{m-1}'$, so that $v'_s$ intersect with an outgoing $b$-edge if and only if $s = \sum_{j=1}^k d_j$ for some $k>0$, and the orientations are from $v'_i$ to $v'_{i+1}$. We let $\sim$ identify $\partial$ to $\partial'$ such that $v_i$ and $v'_{i+c}$ are identified for $i=0,\ldots,m-1$. By the assumption, the image of $\partial$ in $S=\coprod_i S_i/\!\!\sim$ will not have any vertex with two outgoing $b$-edges. There are no incoming $b$-edges at $\partial$ and $\partial'$. Define $\sim$ similarly in the case when $\lambda(\partial)=\lambda^-_{(c_1,\ldots,c_r)}$ and $\lambda(\partial')=\lambda^-_{(d_1,\ldots,d_s)}$. As a result, we get a closed $w$--polygonal surface $S=(\coprod_i S_i)/\!\!\sim$. Hence, $w$ is polygonal.
\end{proof}
\begin{lem}\label{lem:submonoid}
Let $c_1,\ldots,c_r$ be positive integers.
\be
\item
If each $c_i$ is larger than $1$, then $2\lambda^{\pm}_{(c_1,\ldots,c_r)}\in \UU$.
\item
If there exists $1\le i_0\le r$ such that $c_{i_0} \ge 1+\frac12\sum_{i=1}^r c_i$, then
$2\lambda^{\pm}_{(c_1,\ldots,c_r)}\in \UU$.
\ee
\end{lem}
\begin{proof}
	Put $m = \sum_{i=1}^r c_i$.
	
(1) It suffices to show that $\{c_1,c_1+c_2,\ldots,c_1+\cdots+c_r=m\}$ and
$\{1+c_1,1+c_1+c_2,\ldots,1+c_1+\cdots+c_r=1+m\}$ are disjoint modulo $m$ (Definition~\ref{defn:monoid u}).
Assume the contrary. Then for some $1\le h, k\le r$,  we have $\sum_{i=1}^h c_i \equiv 1+\sum_{i=1}^k c_i\mod m)$. Hence, $\sum_{i=1}^h c_i - \sum_{i=1}^k c_i\equiv 1\mod m) $; that is, the sum of some consecutive terms in $c_1,c_2,\ldots,c_r$ is congruent to $\pm1$ modulo $m$. This is a contradiction, since $c_i>1$ for each $i$.

(2)	Without loss of generality, let $i_0=1$. We have $c_1\ge 1+\frac12 m$. In particular, $c_1\ge 2+m-c_1\ge 2$. We claim that $A=\{\sum_{i=1}^h c_i\cobar 1\le h\le r\}$ and $B=\{ (m+1 -c_1) + \sum_{i=1}^k c_i \cobar 1\le k\le r\}$ are disjoint  modulo $m$. Suppose $\sum_{i=1}^h c_i \equiv m+1 -c_1 + \sum_{i=1}^k c_i\ (\textrm{mod }m)$ for some $1\le h,k\le r$. If $h\ge k$ then it follows that $c_1+\sum_{i=k+1}^h c_i \equiv 1 \ (\textrm{mod }m)$, which is impossible since $2\le c_1+\sum_{i=k+1}^h c_i\le \sum_{i=1}^r  c_i=m$. If $h<k$, then $c_1\equiv 1 + \sum_{i={h+1}}^k c_i \ (\textrm{mod }m)$, which is also impossible since $1+\sum_{i=h+1}^k c_i\le 1+\sum_{i=2}^r c_i = 1 + (m - c_1)\le 1+ ( c_1 - 2)=  c_1-1$.
\end{proof}

In Example~\ref{exmp:monoid u},
$\lambda(S) = \lambda(\partial_1)+\lambda(\partial_2)+\lambda(\partial_3)
= 
2\lambda^-_{(2)} + \lambda^+_{(1,1)}$.
So,
$\lambda(S)+\lambda(S) =  4\lambda^-_{(2)}+2\lambda^+_{(1,1)}
=  2\lambda^-_{(2)}+2 ( \lambda^-_{(2)} + \lambda^+_{(1,1)})\in\UU$.
By Lemma~\ref{lem:simple height 1},
we prove again that $w=a(a^2)^b$ is polygonal.

\begin{thm}\label{thm:simple height 1}
	Let $w=\prod_{i=1}^l a^{p_i}(a^{q_i})^b$ be a simple height-one word.
	Put $p=\sum_{i=1}^l |p_i|$ and $q=\sum_{i=1}^l |q_i|$.
	Let $p'$ and $q'$ denote the number of elements in $\{i\cobar p_i=\pm1\}$
	and $\{i\cobar q_i=\pm1\}$, respectively.
	If $pp'\le q^2$ and $qq'\le p^2$, then $w$ is polygonal. 
\end{thm}

\begin{proof}
We will give a proof only for the case when $w$ is a positive height-one word as the other cases are very similar. Without loss of generality, we may assume that $pp'\ge qq'$. Let $c$ be the least common multiple of $\{q_1,q_2,\ldots,q_{l}\}$ and put $d=c!$. Let $P_1,P_2,\ldots,P_c$ ($Q_1,Q_2,\ldots,Q_c$, respectively) be polygonal disks reading $w^{dp}$ ($w^{dq}$, respectively). Define $p_{l+1},p_{l+2},\ldots$ and $q_{l+1},q_{l+2},\ldots$ by $p_{i+l}=p_i$ and $q_{i+l}=q_i$ for each $i$. Especially, $w^k=\prod_{i=1}^{kl} a^{p_i} (a^{q_i})^b$ for any $k>0$. Following the notations in Definition~\ref{defn:consistent}, we let $\alpha_j^{(i)}$ ($\beta_j^{(i)}$, respectively) denote the $b$-edge between intervals $a^{p_j}$ and $a^{q_j}$ ($a^{q_j}$ and $a^{p_{j+1}}$, respectively) on $\partial P_i$. Similarly, let $\gamma_j^{(i)}$ ($\delta_j^{(i)}$, respectively) be the $b$-edge between $a^{p_j}$ and $a^{q_j}$ ($a^{q_j}$ and $a^{p_{j+1}}$, respectively) on $\partial Q_i$. For convention, we let $\beta_0^{(i)} = \beta_{dpl}^{(i)}$ and $\delta_0^{(i)}=\delta_{dql}^{(i)}$. Define a consistent $b$-side-pairing $\sim$ on $(\coprod_i P_i)\coprod (\coprod_i Q_i)$ by
	\bi
	\item
	$\beta_{j-1}^{(i)}\sim\alpha_{j}^{(i)}$ for each $1\le i\le c$ and $1\le j\le dpl$,
	and
	\item
	$\gamma_j^{(i)}\sim\delta_j^{(i)}$ for each $1\le i\le c$ and $1\le j\le dql$.
	\ei
As is easily seen in the example shown in Figure~\ref{fig:simple height 1} (a) and (b),
\begin{eqnarray*}
\lambda(P_i/\!\!\sim)  &=& \sum_{j=1}^{dpl}\lambda^-_{(p_j)} + \lambda^+_{(q_1,q_2,\ldots,q_{dpl})} \\
 &=&  dpp'\lambda^-_{(1)}  
+ d\sum_{1\le j\le pl\atop p_j>1}\lambda^-_{(p_j)}
+\lambda^+_{(q_1,q_2,\ldots,q_{dpl})}
\\
\lambda(Q_i/\!\!\sim)   
 &=&  dqq'\lambda^+_{(1)} 
+ d\sum_{1\le j\le ql\atop {q_j>1}}\lambda^+_{(q_j)}
		+\lambda^-_{(p_1,p_2,\ldots,p_{dql})}.
\end{eqnarray*}	
Let $r=pp'-qq'$. First, consider the case when $r=0$. Put $S = (\coprod_{i=1}^c P_i) \coprod (\coprod_{i=1}^c Q_i)/\!\!\sim$.
\begin{eqnarray*}
2\lambda(S) &=& 2cdpp' (\lambda^-_{(1)}+\lambda^+_{(1)})
+ cd\sum_{1\le j\le pl\atop p_j>1} 2\lambda^-_{(p_j)}
+ cd\sum_{1\le j\le ql\atop q_j>1}  2\lambda^+_{(q_j)}\\
&& + 2c \left(\lambda^+_{(q_1,q_2,\ldots,q_{dpl})} +
 \lambda^-_{(p_1,p_2,\ldots,p_{dql})}\right)
\end{eqnarray*}
Since $q_1+q_2+\cdots+q_{dpl}=dp(q_1+q_2+\cdots+q_l)=dpq = p_1+p_2+\cdots+p_{dql}$, the definition of $\UU$ and Lemma~\ref{lem:submonoid} imply that $2\lambda(S)\in\UU$. By Lemma~\ref{lem:simple height 1}, $w$ is polygonal.
	
Now assume that $r>0$. Since $r\le pp' = |\{j \cobar 1\le j\le pl, p_j=1\}|$, one can choose $A \subseteq \{j\cobar 1\le j\le pl, p_j=1\}$ such that $|A|=r$. Find $x_1,\ldots,x_l\ge0$ satisfying that
	\enumir
	\be
	\item
	$x_j\le q q_j$ and $\sum_j x_j = r$,
	\item
	if $q_j=1$, then $x_j=0$.
	\ee	
	\enumia
	Such $x_j$'s exist, since 
	\[ \sum_{1\le j\le l\atop q_j\ne1} qq_j = q\left( \sum_{1\le j\le l}q_j - \sum_{1\le j\le l\atop q_j=1} q_j \right) = q(q-q') \ge pp'-qq'=r \]
	Fix a map $\sigma\co A\to\{1,\ldots,l\}$ such that $|\sigma^{-1}(j)|=x_j$. We will modify $\sim$ to obtain another consistent $b$-side-pairing $\sim'$ on $(\coprod_i P_i)\coprod(\coprod_i Q_i)$. On $\coprod_i \partial Q_i$, we require that $\sim'$ coincide with $\sim$; that is, $\gamma_j^{(i)}\sim'\delta_j^{(i)}$ for each $1\le i\le c$ and $1\le j\le dql$. We define $\sim'$ on $\coprod_i \partial P_i$ as follows. For $1\le i\le c$ and $1\le j\le dpl$,
	\bi
	\item
	$\beta_{j-1}^{(i)}\sim'\alpha_{j}^{(i+1)}$, if 
	$j\equiv j'\mod pl)$ for some $j'\in A$
	and $i\not\equiv 0\mod q_{\sigma(j')})$,
	\item
	$\beta_{j-1}^{(i)}\sim'\alpha_{j}^{(i+1 - q_{\sigma(j')})}$, if 
	$j\equiv j'\mod pl)$ for some $j'\in A$
	and $i\equiv 0\mod q_{\sigma(j')})$,
	\item
	$\beta_{j-1}^{(i)}\sim'\alpha_{j}^{(i)}$, if 
	$j\not\equiv j'\mod pl)$ for any $j'\in A$.
	\ei
	Put $S' = (\coprod_i P_i)\coprod (\coprod_i Q_i) / \sim'$.
	For each $k\in A$ , define
	$g_k\in \sym(\{1,2,\ldots,c\})$ by
	\[
	g_k(i) = 
	\left\{
	\ba{ll}
		i+1, &\qquad\textrm{if }i\not\equiv 0\mod q_{\sigma(k)}) \\
		i+1-q_{\sigma(k)}, &\qquad\textrm{otherwise.}
	\ea
	\right.
	\]
Then the order of $g_k$ is $q_{\sigma(k)}$. For an arbitrary $1\le i_0\le c$, define $i_1,i_2,\ldots$ by
\[ i_j = \left\{ \ba{ll} 
g_{j'}(i_{j-1}), &\qquad\textrm{if }j\equiv j'\mod pl)\textrm{ for some }j'\in A\\ i_{j-1}, &\qquad\textrm{otherwise} 
\ea \right. \]
Then $\beta_{j-1}^{(i_{j-1})}\sim'\alpha_{j}^{(i_j)}$ for $1\le j\le dpl$. Since $d=c!$, $(\prod_{k\in A}g_k)^d=\mathrm{Id}\in\sym(\{1,2,\ldots,c\})$; hence, $i_{dpl}=(\prod_{k\in A}g_k)^d(i_0)=i_0$. So, $\alpha_{dpl}^{(i_{dpl})}=\alpha_{dpl}^{(i_0)}$. Pick any $1\le i_0\le c$ and denote by $\partial_{i_0}$ the boundary component of $S'$ containing the interval $a^{q_{dpl}}$ from $\partial P_{i_0}= \partial P_{i_{dpl}}$. Since the interval $a^{q_{dpl}}$ from $\partial P_{i_0}$ terminates at the initial vertex of $\beta_{dpl}^{(i_0)}=\beta_0^{(i_0)}$, and $\beta_0^{(i_0)}\sim'\alpha_1^{(i_1)}$, the initial vertex of $a^{q_1}$ from $\partial P_{i_1}$ is concatenated at the terminal vertex of $a^{q_{dpl}}$ from $\partial P_{i_0}$, in $\partial_{i_0}$. In this way, one sees that $\partial_{i_0}$ contains the interval $a_{q_j}$ from $\partial P_{i_j}$ for each $1\le j\le dpl$; for example, see Figure~\ref{fig:simple height 1} (c). It follows that $\lambda(\partial_{i_0}) = \lambda^+_{(q_1,q_2,\ldots,q_{dpl})}$. On the other hand, if $1\le j\le dpl$ and $j\equiv j'\mod pl)$ for some $j'\in A$, then the terminal vertex of the $a$-edge corresponding to $a^{p_j}$ on $\partial P_i$ is concatenated with the initial vertex of $a^{p_j}$ on $\partial P_{g_{j'}^{-1}(i)}$; in particular, the $a$-edges corresponding to $a^{p_j}$ on $\partial P_1,\partial P_2,\ldots, \partial P_c$ form boundary components $\partial_1',\partial_2',\ldots,\partial_{c/q_{\sigma(j')}}'$ in $S'$ where for each $i=1,2,\ldots,c/q_{\sigma(j')}$,
\[\lambda(\partial_i')= \lambda^-_{\underbrace{(1,\ldots,1)}_{q_{\sigma(j')}}}.\]
	
Recall that
\begin{eqnarray*}
\lambda(\coprod_i P_i/\!\!\sim)  
&=&  cdpp'\lambda^-_{(1)} 
+ cd\sum_{1\le j\le pl\atop p_j>1} \lambda^-_{(p_j)}
+c\lambda^+_{(q_1,q_2,\ldots,q_{dpl})}
\end{eqnarray*}
We have
\begin{eqnarray*}
\lambda(\coprod_i P_i/\!\!\sim')   
&=&  cd\sum_{1\le j\le pl\atop p_j=1, j\not\in A}\lambda^-_{(1)} 
+ d\sum_{j\in A} \left(\frac{c}{q_{\sigma(j)}}\right)
\lambda^-_{(\underbrace{1,\ldots,1}_{q_{\sigma(j)}})} 	\\
&&
+ cd\sum_{1\le j\le pl\atop p_j>1} \lambda^-_{(p_j)}
+c\lambda^+_{(q_1,q_2,\ldots,q_{dpl})}\\
&=&
cd(pp'-r)\lambda^-_{(1)} 
+
\sum_{1\le j\le l \atop
q_{j}>1} dx_j\left(\frac{c}{q_{j}}\right)
\lambda^-_{(\underbrace{1,\ldots,1}_{q_{j}})} 	\\
&&
+ cdp\sum_{1\le j\le l\atop p_j>1} \lambda^-_{(p_j)}
+c\lambda^+_{(q_1,q_2,\ldots,q_{dpl})}.
\end{eqnarray*}
\[
	\lambda(\coprod_i Q_i/\!\!\sim') =	\lambda(\coprod_i Q_i/\!\!\sim)    =
	 cdqq'\lambda^+_{(1)} 
+ cdq\sum_{1\le j\le l\atop {q_{j}>1}}\lambda^+_{(q_{j})}
+c\lambda^+_{(p_1,p_2,\ldots,p_{dql})}
\]
Note that $pp'-r=qq'$. From $\lambda(S') = \lambda(\coprod_i P_i/\!\!\sim')+
\lambda(\coprod_i Q_i/\!\!\sim')$, we have:
\begin{eqnarray*}
2\lambda(S')&=& 2cdqq'(\lambda^-_{(1)}+\lambda^+_{(1)}) + \sum_{1\le j\le l \atop q_{j}>1} 2 dx_j\left(\frac{c}{q_{j}}\right) (\lambda^-_{(\underbrace{1,\ldots,1}_{q_{j}})} + \lambda^+_{(q_{j})}) \\
&&+ \sum_{1\le i\le l \atop q_{j}>1} \left(cdq-dx_j\left(\frac{c}{q_{j}}\right)\right) 2\lambda^+_{(q_{j})} \\   
&& + cdp\sum_{1\le j\le l\atop p_j>1} 2\lambda^-_{(p_j)} + c \left(2\lambda^+_{(q_1,q_2,\ldots,q_{dpl})} + 2\lambda^-_{(p_1,p_2,\ldots,p_{dql})}\right).
\end{eqnarray*}
Note that $cdq - dx_i(c/q_i)\ge0$ by the choice of $x_i$.
Hence, $2\lambda(S')\in\UU$. In particular, $w$ is polygonal.
\end{proof}

\begin{figure}[htb!]
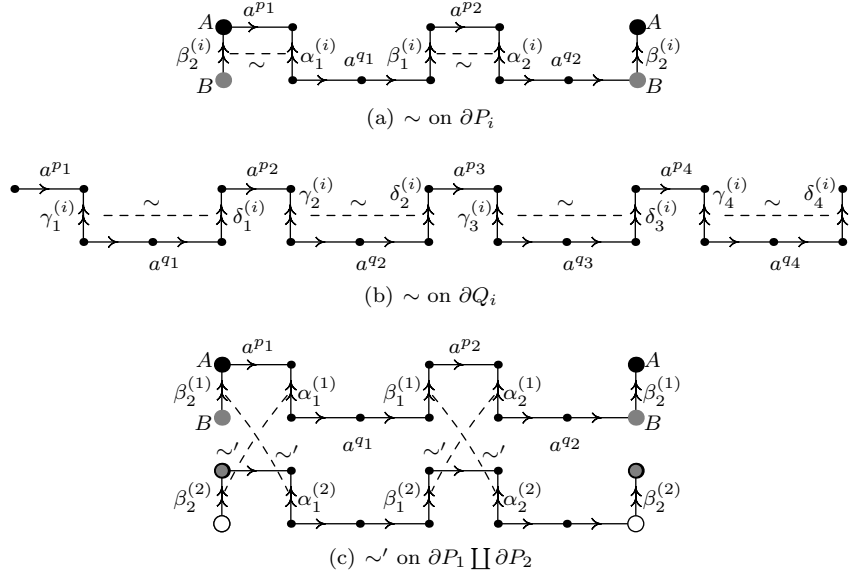

	\centering
	\subfigure[$\sim$ on $\partial P_i$]{
	\includegraphics[]{fig/figprimep.0}
	}
	\subfigure[$\sim$ on $\partial Q_i$]{
	\includegraphics[]{fig/figprimeq.0}
	}
	\subfigure[$\sim'$ on $\partial P_1\coprod\partial P_2$]{
	\includegraphics[]{fig/figprimesim.0}
	}
	\caption{Example~\ref{exmp:bs2} (1).
	The orientations and the labels of the edges of
	$\partial P_i$ and $\partial Q_i$.
	In (a) and (c), the vertices labeled by the same alphabets are actually the same 
	in $P_i$.
	\label{fig:simple height 1}}
\end{figure}

\begin{exmp}
	\be
	\item
Let us illustrate the proof of Theorem~\ref{thm:simple height 1} for the Baumslag--Solitar relator $w=a (a^2)^b$. Note $p=1=p', q=2,q'=0$ and $c=2=d$. As in Figure~\ref{fig:simple height 1} (a) and (b), choose $P_1,P_2,Q_1,Q_2$ such that $\partial P_i$ reads $w^{cp}=w^2$ and $\partial Q_i$ reads $w^{cq} = w^4$. On $P_i$ and $Q_i$, a consistent $b$-side-pairing $\sim$ is shown by dashed lines. Then, $\lambda(P_1/\!\!\sim) = \lambda(P_2/\!\!\sim) = 2\lambda^-_{(1)} + \lambda^+_{(2,2)}$ and $\lambda(Q_1/\!\!\sim) = \lambda(Q_2/\!\!\sim) = \lambda^-_{(1,1,1,1)} + 4\lambda^+_{(2)}$. Let a new partial side-pairing $\sim'$ coincide with $\sim$ on $\coprod \partial Q_i$ and be modified on $\coprod \partial P_i$ as shown in Figure~\ref{fig:simple height 1} (c). We have $\lambda(P_1\coprod P_2/\!\!\sim') = 2\lambda^-_{(1,1)} + 2\lambda^+_{(2,2)}$. Define $S' = (\coprod_i P_i)\coprod (\coprod_i Q_i)/\!\!\sim'$. $2\lambda(S') = 4\lambda^-_{(1,1)} + 4\lambda^+_{(2,2)} + 4\lambda^-_{(1,1,1,1)} + 16\lambda^+_{(2)} = 4 (\lambda^-_{(1,1)}+\lambda^+_{(2)}) + 4 (\lambda^+_{(2,2)} + \lambda^-_{(1,1,1,1)}) + 6(2\lambda^+_{(2)})\in \UU$.
	\item
	More generally, any Baumslag--Solitar relator $w = a^p (a^q)^b$ is a simple height-one word for $pq\ne0$. As $pp'$ and $qq'$ are either $0$ or $1$, $pp'\le1\le q^2$ and $qq'\le1\le p^2$. By Theorem~\ref{thm:simple height 1}, $w$ is polygonal.
	\ee\label{exmp:bs2}
\end{exmp}

\begin{cor}\label{cor:simple height 1}
	Let $w=\prod_{i=1}^l a^{p_i}(a^{q_i})^b$ be a simple height-one word. Put $p=\sum_{i=1}^l |p_i|$ and $q = \sum_{i=1}^l |q_i|$. If there exists $r>1$ such that $\min(p/l,q/l)\ge r$ and $1/r \le p/q\le r$, then $w$ is polygonal.
\end{cor}

\begin{proof}
Note that $p' = \sum_{|p_i|=1} 1\le l$, and similarly, $q'=\sum_{|q_i|=1}1\le l$. $pp'\le pl \le rq \cdot (q/r) =q^2$, and by symmetry, $qq'\le p^2$. By Theorem~\ref{thm:simple height 1}, $w$ is polygonal.
\end{proof}

\begin{rem}\label{rem:simple height 1}
Let $N$ be a sufficiently large integer. Consider the probability space $X_N$ of the positive words of length $N$ in $F_2=\langle a,b\rangle$, where $a$ and $b$ are equally likely to appear. Set $f$ to be the injective endomorphism on $F_2$ defined by $f(a)=a$ and $f(b)=a^b$. Let $Y_N=f(X_N)$ have the push-forward  probability from $X_N$ by $f$.  Write an arbitrary word $w\in Y_N$ as $w=f(\prod_{i=1}^{m}a^{c_i}b^{d_i}) = \prod_{i=1}^{m}a^{c_i}(a^{d_i})^b$, where $m>0$, $c_i,d_i\ge0$ and moreover, $c_i,d_{m+1-i}\ne0$ unless $i=1$. Let $s$ denote the number of \textit{runs} in $w$; this means, $s$ is the number of non-trivial terms in $(a^{c_1},b^{d_1},\ldots,a^{c_m},b^{d_m})$. 
If $w=a^N$ or $w=(a^N)^b$, we put $(p_1,q_1)=(N,0)$ or $(p_1,q_1)=(0,N)$ respectively, and also let $w'= w$ and $l=\frac12$; otherwise, find a cyclic conjugation $w'\sim w$ such that $w'=\prod_{i=1}^{l}a^{p_i}(a^{q_i})^b$ where $p_i,q_i>0$ for all $i$. 
Define $p=\sum_i p_i,q=\sum_i q_i,p'=\sum_{|p_i|=1}1$ and $q'=\sum_{|q_i|=1}1$. Note that $p,q,p',q',l$ and $s$ are well-defined by $w$, and that $p$ and $q$ have the binomial distribution $B(N,1/2)$. It is obvious that $p',q'\le l\le \frac12 s$ and so, \[P(q^2\le pp')\le P(q^2 \le  \frac12 ps).\] 
Let us set $\epsilon=\frac16$ and $\alpha = (({1-\epsilon})/({1+\epsilon}))^2 = {25}/{49}>1/2$. 
Denote by $A$ the event that $\frac{1-\epsilon}2 N\le q\le \frac{1+\epsilon}2 N$.
If $A$ occurs, then $2q^2/p = \frac{q}{N-q}\cdot 2q \ge \frac{1-\epsilon}{1+\epsilon}\cdot(1-\epsilon)N$. So,
\begin{eqnarray*}
P(q^2\le\frac12 ps) &=& P((2q^2/p \le s) \textrm{ and } A) + P((2q^2/p \le s)\textrm{ and } A^c)\\
&\le& P\left( \left(\frac{(1-\epsilon)^2}{1+\epsilon} N \le s\right) \textrm{ and } A\right)
+ P(A^c) \\
&\le& P\left(s\ge \frac{(1-\epsilon)^2}{1+\epsilon} N\right)  + P(A^c)\le P(s-1 \ge \alpha(N-1))+ P(A^c).
\end{eqnarray*}
The last inequality follows from the fact that
\[
\frac{(1-\epsilon)^2}{1+\epsilon} N -1 \ge \left(\frac{1-\epsilon}{1+\epsilon}\right)^2 (N - 1) = \alpha(N-1)
\]
for sufficiently large $N$.  Since $s-1$ has the binomial distribution $B(N-1,1/2)$~\cite{Mood:1940p5063},
the expectation and the standard deviation of $s-1$ are $(N-1)/2$ and $\sqrt{N-1}/2$, respectively.  By Chebyshev's inequality,
\begin{eqnarray*}
	 &&P(s-1 \ge \alpha(N-1)) =P\left( s-1-\frac{N-1}2
		\ge (\alpha-\frac12)(N-1)\right) \\
	 &=& P\left( \frac{s-1-(N-1)/2}{\sqrt{N-1}/2} 
		\ge 2(\alpha-\frac12)\sqrt{N-1}\right)
\le \frac1{(2\alpha-1)^2(N-1)}
\end{eqnarray*}
and
\begin{eqnarray*}
&& P(A^c) = P(|q-N/2| > \epsilon N/2) \\
&=& P\left(\frac{|q-N/2|}{\sqrt{N}/2} > \epsilon\sqrt{N}\right)
\le \frac1{\epsilon^2 N}\le\frac1{\epsilon^2(N-1)}.  
\end{eqnarray*}
Therefore, $P(q^2\le pp')\le C/(N-1)$ for some constant $C>0$. In particular, $P(q^2\le pp')$ converges to $0$ as $N$ approaches infinity.
By symmetry, $P(p^2\le qq')\to0$ as $N\to\infty$.
From Theorem~\ref{thm:simple height 1}, it follows that the probability for a word in $Y_N$ to be polygonal converges to $1$ as $N\to \infty$. The same argument applies to the cases when $c_i<0$ or $d_i<0$.
In this sense, one concludes that simple height-one words are \textit{almost surely} polygonal.
\end{rem}

\bibliographystyle{plain}

\begin{thebibliography}{10}

\bibitem{Bestvina:1992p456}
M~Bestvina and M~Feighn.
\newblock A combination theorem for negatively curved groups.
\newblock {\em J. Differential Geom.}, 35(1):85--101, 1992.

\bibitem{BridsonHaefliger}
Martin~R. Bridson and Andr{\'e} Haefliger.
\newblock {\em Metric spaces of non-positive curvature}, volume 319 of {\em
  Grundlehren der Mathematischen Wissenschaften [Fundamental Principles of
  Mathematical Sciences]}.
\newblock Springer-Verlag, Berlin, 1999.

\bibitem{Calegari:2008p1810}
D~Calegari.
\newblock Surface subgroups from homology.
\newblock {\em Geometry {\&} Topology}, 12(4):1995--2007, 2008.

\bibitem{Gordon:2009p360}
C~McA Gordon and H~Wilton.
\newblock On surface subgroups of doubles of free groups.
\newblock {\em J. Lond. Math. Soc. (2)}, 82(1):17--31, 2010.

\bibitem{Haglund}
Fr{\'e}d{\'e}ric Haglund.
\newblock Finite index subgroups of graph products.
\newblock {\em Geom. Dedicata}, 135:167--209, 2008.

\bibitem{mccammond:personal}
Jon McCammond.
\newblock Personal communication, 2010.

\bibitem{Mood:1940p5063}
A.~M Mood.
\newblock The distribution theory of runs.
\newblock {\em Ann. Math. Statistics}, 11:367--392, 1940.

\bibitem{Scott:1978p142}
P~Scott.
\newblock Subgroups of surface groups are almost geometric.
\newblock {\em J. London Math. Soc. (2)}, 17(3):555--565, 1978.

\bibitem{SW1979}
P~Scott and T~Wall.
\newblock Topological methods in group theory. \textit{Homological group theory
  ({P}roc. {S}ympos., {D}urham, 1977)}, 137--203, {L}ondon {M}ath. {S}oc.
  {L}ecture {N}ote {S}er., 36, {C}ambridge {U}niv. {P}ress, {C}ambridge-{N}ew
  {Y}ork, 1979.

\bibitem{Stallings:1983p596}
J~R Stallings.
\newblock Topology of finite graphs.
\newblock {\em Invent. Math.}, 71(3):551--565, 1983.

\bibitem{wilton}
Henry Wilton.
\newblock Hall's theorem for limit groups.
\newblock {\em Geom. Funct. Anal.}, 18(1):271--303, 2008.

\bibitem{Wise:2000p790}
D~T Wise.
\newblock Subgroup separability of graphs of free groups with cyclic edge
  groups.
\newblock {\em Q. J. Math.}, 51(1):107--129, 2000.

\end{thebibliography}

\end{document}